\documentclass[twoside]{amsart}
\usepackage{amsmath,amssymb,amscd,amsthm,amsxtra,amsfonts}
\usepackage{mathtools,mathrsfs,dsfont,xparse}
\usepackage{graphicx,epstopdf}
\usepackage{multirow}
\usepackage{cases}
\usepackage{enumerate}
\usepackage{xcolor}
\usepackage{tikz,pgfplots}
\usetikzlibrary{matrix}
\usepackage{mathtools}
\usepackage[margin=1 in]{geometry}
\mathtoolsset{showonlyrefs}
\usepackage[numbers,sort]{natbib}
\newtheorem{theorem}{Theorem}[section]
\newtheorem{lemma}[theorem]{Lemma}
\newtheorem{proposition}[theorem]{Proposition}
\newtheorem{corollary}[theorem]{Corollary}

\newcommand{\R}{\mathbb{R}}
\newcommand{\T}{\mathbb{T}}

\newcommand{\beq}{\begin{equation}}
\newcommand{\eeq}{\end{equation}}
\newcommand{\beqq}{\begin{equation*}}
\newcommand{\eeqq}{\end{equation*}}

\theoremstyle{definition}

\theoremstyle{remark}
\newtheorem{remark}[theorem]{Remark}

\numberwithin{equation}{section}

\usepackage[colorlinks=true, pdfstartview=FitV, linkcolor=blue, citecolor=blue, urlcolor=blue]{hyperref}

\DeclareMathOperator{\re}{Re}
\DeclareMathOperator{\im}{Im}

\newcommand{\abs}[1]{\left|#1\right|}
\newcommand{\norm}[1]{\left\|#1\right\|}
\newcommand{\inner}[1]{\left \langle#1\right \rangle}

\newcommand{\parenthese}[1]{\left(#1\right)}

\numberwithin{equation}{section}

\begin{document}

\title[]{Singular Levy processes and dispersive effects of generalized Schr\"odinger equations}
\author{Yannick Sire, Xueying Yu, Haitian Yue and Zehua Zhao}

\address{Yannick Sire
\newline \indent Department of Mathematics, Johns Hopkins University\indent 
\newline \indent  Krieger Hall, 3400 N. Charles St.,
Baltimore, MD, 21218.\indent }
\email{ysire1@jhu.edu}

\address{Xueying  Yu
\newline \indent Department of Mathematics, University of Washington\indent 
\newline \indent  C138 Padelford Hall Box 354350, Seattle, WA 98195,\indent }
\email{xueyingy@uw.edu}

\address{Haitian Yue
\newline \indent Institute of Mathematical Sciences, ShanghaiTech University\newline\indent
Pudong, Shanghai, China.}
\email{yuehaitian@shanghaitech.edu.cn}

\address{Zehua Zhao
\newline \indent Department of Mathematics and Statistics, Beijing Institute of Technology, Beijing, China.
\newline \indent MIIT Key Laboratory of Mathematical Theory and Computation in Information Security, Beijing, China.}
\email{zzh@bit.edu.cn}




\begin{abstract}
We introduce new models for Schr\"odinger-type equations, which generalize standard NLS  and for which different dispersion occurs depending on the directions. Our purpose is to understand dispersive properties depending on the directions of propagation, in the spirit of waveguide manifolds, but where the diffusion is of different types. We mainly consider the standard Euclidean space and the waveguide case but our arguments extend easily to other types of manifolds (like product spaces). Our approach unifies in a natural way several previous results. Those models are also generalizations of some appearing in seminal works in mathematical physics, such as relativistic strings.  In particular, we prove the large data scattering on waveguide manifolds $\mathbb{R}^d \times \mathbb{T}$, $d \geq 3$. This result can be regarded as the analogue of \cite{TV2, YYZ2} in our setting and the waveguide analogue investigated in  \cite{GSWZ}. A key ingredient of the proof is a Morawetz-type estimate for the setting of this model.
\end{abstract}

\maketitle



\setcounter{tocdepth}{1}
\tableofcontents

\parindent = 10pt     
\parskip = 8pt

\section{Introduction}

Anomalous diffusion, also called fractional diffusion,  appears  naturally in the physics  and mathematical physics literature in the study of relativistic matter and strings, see e.g. for the works \cite{carmona,daubechies,FLS1,FLS2,LiebYau1,LiebYau2} and references therein. Similarly, a tentative description of some quantum mechanics has been undertaken in  \cite{Laskin,Laskin2,Laskin3}. The classical relativistic operator is the multiplier $\sqrt{|\xi|^2 +m^2}-m$ where $m\geq 0$ is a constant. Fractional Schr\"odinger operators are a fundamental equation of fractional quantum mechanics, which was derived by Laskin \cite{Laskin3} as a result of extending the Feynman path integral, from the Brownian-like to L\'evy-like quantum mechanical paths. The corresponding physical realizations  were made in condensed matter physics \cite{St13} and in nonlinear optics \cite{Lon15}.

In the present contribution, we investigate dispersive properties of Schr\"odinger operators generalizing the previous operator. One possible generalization is to replace the square root with any power between $0$ and $1$ (see e.g. \cite{carmona} for a similar generalization). Those operators have been investigated recently in different directions. The long-time behaviors (such as global well-posedness, scattering, blow-up, and the existence of invariant measures) of the solutions are interesting and widely studied. 
In \cite{BHL}, the blow-up with radial data in certain regimes was constructed by deriving a localized virial estimate for the fractional Schr\"odinger equation.
In \cite{GSWZ}, the first author of this paper together with Guo, Wang and Zhao  performs  the Kenig-Merle’s concentration-compactness-rigidity method \cite{KM06} and obtains global well-posedness and scattering in the energy space in the defocusing case, and in the focusing case with energy below the ground state. We also refer the reader to \cite{GH11,CGKL13,KLR13,IP14,HS,SWTZ18,ST20,SW21,ST21, DET, SY1, SY2} and
references therein for many other results on the long time behaviors for relativistic NLS. In all the previous works the operator under consideration is $(-\Delta)^\sigma$ for $\sigma \in (0,1)$, which is the multiplier $|\xi|^\sigma$. The ambiant space is either the whole Euclidean space $\mathbb R^d$ or the (rational) torus $\T^d$. The aim is two-fold: first, we consider the waveguide manifold, i.e. $(-\Delta)^\sigma$ defined on $\R^d \times \T$; more interestingly, we introduce a new model where the operator acts differently according to the spatial direction. 

In the following, we consider only the case of the waveguide $\R^d \times \T^n$. To motivate the reason behind the equations we consider, we first state the Levy-Kintchine formula (see e.g. \cite{bertoin}): every (isotropic) Levy process with pure jumps in $\R^d$  is given by the multiplier 
\begin{align}
\mathfrak{m(\xi)}=\int_{\R^d}\parenthese{ 1-e^{i\,\xi \cdot x} +i\,\xi \cdot x \chi_{|x| <1} }\Pi(dx)
\end{align}
where $\Pi(dx)$ is the so-called Levy measure satisfying the integrability condition 
\begin{align}
\int_{\R^d} \max(1, |x|^2)\Pi(dx) <\infty .
\end{align}
The multiplier $\mathfrak{m}(\xi)=|\xi|^{2\sigma}$ is obtained by choosing $\Pi(dx)=\frac{1}{|x|^{d +2\sigma}}\,dx$. On the other hand, one could consider singular Levy measures just as  
Dirac masses supported along some (or all) axis coordinates.  In this paper, we consider the following model 
\begin{align}\label{gNLShomog}
\begin{cases}
i\partial_t u+\parenthese{(-\Delta_x)^\sigma+(-\partial^2_y)^\sigma } u=F, & (x,y)\in \R^d \times \T^n, \\
u(0)= u_0\in H^{\sigma} (\R^d \times \T^n). &
\end{cases}
\end{align}

In view of the previous discussion, the operator $(-\Delta_x)^\sigma+(-\partial^2_y)^\sigma $ corresponds to the multiplier $\mathfrak{m}(\xi,\eta)=|\xi|^{2\sigma}+|\eta|^{2\sigma}$ where $\xi$ is the Fourier variable corresponding to $x \in \R^d$ and $\eta$ the one to $y \in \T^n$. In particular, this is a Levy process with a singular Levy measure. To the best of our knowledge, dispersive and space-time estimates for such propagators have not been considered yet in the literature. 

\begin{remark}
It is important to notice that all the (dual) variables in $\mathbb R^d \times \T^n$ have to be present in the multiplier $\mathfrak{m}$. Otherwise, the associated operator  becomes subelliptic, and major dispersive issues arise as in the case of the Schr\"odinger propagator in the Heisenberg group \cite{BGX} or the Szeg\"o model of G\'erard and Grellier (see e.g. \cite{GG} and subsequent articles). 
\end{remark}

The case of {\sl regular} Levy measures has been considered for instance in \cite{CHKL,GSWZ,HS}. We would like to comment further on the model under consideration: It is a well-known fact that on a compact manifold, there is no dispersion, and loss of derivatives occurs in the Strichartz estimates (see e.g. \cite{BGT}). Similarly, on a waveguide manifold (or a product manifold), e.g. $\mathbb R^d \times \mathbb T^n \subset \mathbb R^{d +n}$, dispersion occurs but the fact that only part of the directions ($d$ here) contributes to diffusion introduces several complications. Our model is another instance of such a phenomenon: it is by now well-known that anomalous diffusion of order $2\sigma $ exhibits a loss of $1-\sigma$. A variation on our model would be for example to consider the Euclidean factors $\R^d \times \R$ but considering a different diffusion on each factor $\R^d$ and $\R$.  Product spaces $\mathbb{R}^d\times \mathbb{T}^m$ are known as `waveguide manifolds' and are of particular interest in nonlinear optics. We refer to \cite{CZZ,CGZ,HP,HTT1,HTT2,IPT3,IPRT3,KV1,YYZ,Z1,Z2,ZhaoZheng} with regard to  waveguide settings.

An interesting feature of the equation \eqref{gNLS} is its product structure. In particular, this allows us to get interaction Morawetz estimates using the tensorization argument in \cite{CKSTT} and also use in a much better-streamlined way the vector-valued argument of Tzvetkov-Visciglia \cite{TV1}. This latter operator satisfies the non-degeneracy assumptions in Schippa too \cite{Schippa}. One could also consider different geometries in the ambiant space like the pure tori case $\T^d \times \T$ or the pure Euclidean one $\R^d \times \R$. We chose to consider the waveguide case as a middle point between these two geometries. It is relatively straightforward to generalize our results to those two geometries. In the case of the classical fractional laplacian on pure tori, global well-posedness has been considered by Schippa in \cite{Schippa} using a decoupling approach. 
As far as nonlinear applications are concerned, we will be considering (unless otherwise stated we will always consider $n=1$)
\begin{align}\label{gNLS}
\begin{cases}\tag{gNLS}
i\partial_t u+ \parenthese{(-\Delta_x)^\sigma+(-\partial^2_y)^\sigma } u=\mu|u|^pu, & (x,y)\in \R^d \times \T, \\
u(0)= u_0\in H^{\sigma} (\R^d \times \T), &
\end{cases}
\end{align}

where $\mu=\pm 1$ and $\frac{4\sigma}{d}<p<\frac{4\sigma}{d+1-2\sigma}$.

\begin{remark}
We consider the exponent $p$ to be in the subcritical range, for some technical reasons. The left endpoint indicates mass-critical if we ignore the torus direction; the right endpoint indicates energy-critical if we regard the torus direction as Euclidean direction. So essentially, the problem is energy-subcritical and mass-supercritical. 
\end{remark}

The well-posedness theory and the long-time dynamics for generalized Schr\"odinger operators on waveguide manifolds are understudied. It is our goal here to fill this gap in the literature by exhibiting a wealth of different techniques previously used to deal with some models in the case here of waveguides. For the extremal cases of tori or Euclidean spaces, we refer the reader to \cite{CHKL,Dinh,GSWZ,HS} and references therein.

The operators under consideration are well-designed to generalize Tzvetkov-Visciglia's results in \cite{TV1,TV2} thanks to the product structure of the operator.  Our main result is 
\begin{theorem}\label{mainthm}
Let $\sigma>\frac{1}{2}$ in \eqref{gNLS} and assume the spatial variable is radial  in $\mathbb{R}^d$. Then we have:
\begin{itemize}

\item{(i)} for any initial datum $u_0\in H_{x,y}^{\sigma}$, the IVP \eqref{gNLS} has a unique local solution $u(t,x,y)\in \mathcal{C}((-T,T);H_{x,y}^{\sigma})$ where $T=T(\|u_0\|_{H_{x,y}^{\sigma}})>0$; 

\item{(ii)} Moreover, when $\mu=-1$ (defocusing case), the solution $u(t, x, y)$ can be extended globally in time. Moreover, the solution scatters in the following sense: there exist $f^{\pm} \in H^{\sigma}_{x,y}(\mathbb{R}^d\times \mathbb{T})$ such that
\begin{equation}\label{scattering}
    \lim_{t \rightarrow \pm \infty} \|u(t,x,y)-e^{it \parenthese{(-\Delta_x)^\sigma + (-\partial^2_y)^\sigma }}f^{\pm}\|_{H^{\sigma}_{x,y} (\mathbb{R}^d\times \mathbb{T})} =0.
\end{equation}
 \end{itemize}
\end{theorem}

We comment on the previous results: 

\begin{enumerate}
\item The radiality assumption is to avoid loss of derivatives in the Strichartz estimates. For our range of powers, this is a technical assumption that can be removed. 
\item The assumption $\sigma>\frac{1}{2}$ is also due to some technical reasons if one wants to modify the method in Tzvetkov-Visciglia \cite{TV2}. Roughly speaking, it is because the Sobolev embedding exponent in 1D is $\frac{1}{2}+$. 
\item We provide two different proofs of the well-posedness result: one using decoupling and one using vector-valued analysis. 
\end{enumerate}

\begin{remark}
Since NLS with a harmonic trapping potential has similar properties/behaviors as NLS on tori, heuristically one may compare NLS with a partial harmonic potential with NLS on waveguides. (See \cite{ACS,Guo5,HT} and the references therein.) Thus one may conjecture that it is possible to obtain the analogue of Theorem \ref{mainthm} with a partial harmonic potential. 
\end{remark}

The paper is organized as follows: In section \ref{sec Pre}, we discuss preliminaries including notations and some useful estimates; in Section \ref{sec LWP1}, we discuss well-posedness theory  using vector-valued argument; in Section \ref{sec LWP2}, we discuss well-posedness theory  using decoupling argument; in Section \ref{sec Morawetz}, we establish a Morawetz-type estimate for our equation, which is the crucial step for obtaining the decay property of solutions of \eqref{gNLS}; in Section \ref{sec Scattering}, we give the proof for the large data scattering; in Section \ref{sec Future}, we give a few more remarks on the research line of `dispersive equations on waveguide manifolds'.

\subsection*{Acknowledgment.}Y. S. is partially supported by the Simons foundation through a Simons collaborative grant for mathematicians and NSF grant DMS-2154219. X. Y. was funded in part by an AMS-Simons travel grant. H. Y. was supported by a start-up funding of ShanghaiTech University. Z. Z. was supported by the NSF grant of China (No. 12101046) and the Beijing Institute of Technology Research Fund Program for Young Scholars.

\section{Preliminaries}\label{sec Pre}
In this section, we briefly discuss notations and some basic estimates.
\subsection{Notations}
We write $A \lesssim B$ to say that there is a constant $C$ such that $A\leq CB$. We use $A \simeq B$ when $A \lesssim B \lesssim A $. Particularly, we write $A \lesssim_u B$ to express that $A\leq C(u)B$ for some constant $C(u)$ depending on $u$. 

Then we give some more preliminaries on the setting of the waveguide manifold. The tori case can be defined similarly. In fact, it is included since it is a special case. Throughout this paper, we regularly refer to the spacetime norms
\begin{equation}
    \|u\|_{L^p_tL^q_z(I_t \times \mathbb{R}^m\times \mathbb{T}^n)}=\left(\int_{I_t}\left(\int_{\mathbb{R}^m\times \mathbb{T}^n} |u(t,z)|^q \, dz \right)^{\frac{p}{q}} \, dt\right)^{\frac{1}{p}}.
\end{equation}
Similarly, we can define the composition of three $L^p$-type norms like $L^p_tL^q_xL^2_y$. Moreover, we turn to the Fourier transformation and Littlewood-Paley theory. We define the Fourier transform on $\mathbb{R}^m \times \mathbb{T}^n$ as follows:
\begin{equation}
    (\mathcal{F} f)(\xi)= \int_{\mathbb{R}^m \times \mathbb{T}^n}f(z)e^{-iz\cdot \xi} \, dz,
\end{equation}
where $\xi=(\xi_1,\xi_2,...,\xi_{d})\in \mathbb{R}^m \times \mathbb{Z}^n$ and $d=m+n$. We also note the Fourier inversion formula
\begin{equation}
    f(z)=c \sum_{(\xi_{m+1},...,\xi_{d})\in \mathbb{Z}^n} \int_{(\xi_1,...,\xi_{m}) \in \mathbb{R}^m} (\mathcal{F} f)(\xi)e^{iz\cdot \xi} \, d\xi_1...d\xi_m.
\end{equation}
For convenience, we may consider the discrete sum to be an integral with the discrete measure so we can combine the above integrals together and treat them to be one integral. Moreover, we define the Schr{\"o}dinger propagator $e^{it\Delta}$ by
\begin{equation}
    \left(\mathcal{F} e^{it\Delta}f\right)(\xi)=e^{-it|\xi|^2}(\mathcal{F} f)(\xi).
\end{equation}
Similarly, for more general operator $e^{it\phi(\nabla/i)}$,
\begin{equation}
    \left(\mathcal{F} e^{it\phi(\nabla/i)} f\right)(\xi)=e^{-it\phi(\xi)}(\mathcal{F} f)(\xi).
\end{equation}
We are now ready to define the Littlewood-Paley projections. First, we fix $\eta_1: \mathbb{R} \rightarrow [0,1]$, a smooth even function satisfying
\begin{equation}
    \eta_1(\xi) =
\begin{cases}
1, \ |\xi|\le 1,\\
0, \ |\xi|\ge 2,
\end{cases}
\end{equation}
and $N=2^j$ a dyadic integer. Let $\eta^d=\mathbb{R}^d\rightarrow [0,1]$, $\eta^d(\xi)=\eta_1(\xi_1)\eta_1(\xi_2)\eta_1(\xi_3)...\eta_1(\xi_d)$. We define the Littlewood-Paley projectors $P_{\leq N}$ and $P_{ N}$ by
\begin{equation}
    \mathcal{F} (P_{\leq N} f)(\xi):=\eta^d\left(\frac{\xi}{N}\right) \mathcal{F} (f)(\xi), \quad \xi \in \mathbb{R}^m \times \mathbb{Z}^n,
\end{equation}
and
\begin{equation}
P_Nf=P_{\leq N}f-P_{\leq \frac{N}{2}}f.
\end{equation}
For any $N\in (0,\infty)$, we define
\begin{equation}
    P_{\leq N}:=\sum_{M\leq N}P_M,\quad P_{> N}:=\sum_{M>N}P_M.
\end{equation}

\subsection{Some useful estimates}
In this subsection, we discuss some useful estimates to prove Theorem \ref{mainthm}. We note that there is a symmetric assumption for the initial data in Theorem \ref{mainthm}, but for the sake of completeness, we also discuss the nonradial Strichartz estimate. We refer to Tao \cite{Taobook} for the standard Strichartz estimate for NLS and refer to Guo-Wang \cite{GW} and the references therein for general radial Stricharz estimates.

We say that $(p,q)$ is admissible if
\begin{equation}
    \frac{2}{p}+\frac{d}{q}=\frac{d}{2},\quad 2 \leq p,q \leq \infty \quad (p,q,d)\neq (2,\infty,2).
\end{equation}
For the nonradial case, we define the following Strichartz norm 
\begin{equation}
    \|u\|_{S^{s}_{p,q}}:=\||\nabla|^{-d(1-\sigma)(\frac{1}{2}-\frac{1}{q})}u\|_{L^p_{t\in I}W_{x}^{s,q}}
\end{equation}
where $I=[0,T)$. We use $u_0$ for the initial data and $F$ for the nonlinearity. Consider
\begin{equation}\label{FNLS2}
    (i\partial_t+(-\Delta_{x})^{\sigma})u=F,\quad u(0)= u_0\in H^{\sigma}(\mathbb{R}^d).
\end{equation}
Then we have the Strichartz estimate (see \cite{CHKL,HS}),
\begin{lemma}[Strichartz estimate]\label{nonradialStri}
Fix $\sigma \in (0,1)$ and $s >0$. For an admissible pair $(p,q)$ and $(a,b)$, we
have
\begin{equation}
   \|e^{it(-\Delta)^{\sigma}} u_0\|_{S^{s}_{p,q}} \lesssim \|u_0\|_{H_x^s}
\end{equation}
and
\begin{equation}
    \|\int_0^t e^{i(t-s)(-\Delta)^{\sigma}} F(s)ds\|_{S^{s}_{p,q}} \lesssim \||\nabla|^s F\|_{L_{t\in I}^{a'}L_x^{b'}}.
\end{equation}
\end{lemma}
\begin{corollary}
Consider $(\tilde{p},\tilde{q})$ satisfying 
\begin{equation}
    \frac{2}{\tilde{p}}+\frac{d}{\tilde{q}}=\frac{d}{2}-s+d(1-\sigma)(\frac{1}{2}-\frac{1}{q}),
\end{equation}
where $s \geq d(1-\sigma)(\frac{1}{2}-\frac{1}{q})$. Consider  $(a,b)$ to be admissible pair. Then
\begin{equation}
   \|e^{it(-\Delta)^{\sigma}} u_0\|_{L^{\tilde{p}}_{t}L^{\tilde{q}}_{x}} \lesssim \|u_0\|_{H_x^s}
\end{equation}
and
\begin{equation}
    \|\int_0^{t}e^{i(t-s)(-\Delta)^{\sigma}} F(s)ds\|_{L^{\tilde{p}}_{t}L^{\tilde{q}}_{x}} \lesssim \||\nabla_x|^s F\|_{L_{t\in I}^{a'}L_x^{b'}}.
\end{equation}
\end{corollary}
For proving Theorem \ref{mainthm}, we introduce the radial Strichartz estimate for $\sigma$-admissible pairs (see Lemma 2.1 of \cite{GSWZ}). Consider $\sigma>\frac{1}{2}$ and  $u$ solving \eqref{FNLS2}, then the following estimate holds
\begin{equation}
    \|u\|_{L^p_tL^q_x}\lesssim \|u_0\|_{\dot{H}^{\gamma}}+\|F\|_{L^{\tilde{p}'}_tL^{\tilde{q}'}_x},
\end{equation}
where the pairs satisfy
\begin{equation}\label{rela}
    \frac{2\sigma}{p}+\frac{d}{q}=\frac{d}{2}-\gamma, \quad \frac{2\sigma}{\tilde{p}}+\frac{d}{\tilde{q}}=\frac{d}{2}+\gamma.
\end{equation}
In the exponent relation \eqref{rela}, we call $(p,q)$ is $\sigma$-admissible with $\gamma$ regularity; we call $(p,q)$  $\sigma$-admissible if $\gamma=0$.

The previous estimates hold on Euclidean spaces of dimension $d$. On the waveguide $\mathbb{R}^d\times \mathbb{T}$, one can adopt the strategy in \cite{TV1} (Proposition $2.1$ for the NLS case) based on mixed norms (or in other words vector-valued norms). We discuss this case now. We prove
\begin{lemma}\label{FNLSStrichartz}
Consider \eqref{gNLShomog} and pairs $(p,q),(\tilde{p},\tilde{q})$ satisfying the exponent relations \eqref{rela}, then
\begin{equation}\label{rStri1}
    \|u\|_{L^p_tL^q_xL^2_y}\lesssim \|u_0\|_{\dot{H}^{\gamma}}+\|F\|_{L^{\tilde{p}'}_tL^{\tilde{q}'}_xL^2_y},
\end{equation}
and
\begin{equation}\label{rStri2}
    \|u\|_{L^p_tL^q_xH^{\gamma}_y}\lesssim \|u_0\|_{\dot{H}^{\gamma}}+\|F\|_{L^{\tilde{p}'}_tL^{\tilde{q}'}_xH^{\gamma}_y}.
\end{equation}
\end{lemma}
\begin{proof}[Proof of Lemma \ref{FNLSStrichartz}]
The proof of Lemma \ref{FNLSStrichartz} is very similar to Proposition 2.1 of \cite{TV1} (NLS case) so we just explain the difference from the NLS case here. For the NLS case, the main idea is decomposing the functions with respect to the orthonormal basis of $L^2(\mathbb{T})$ given by the eigenfunctions $\{\phi_j\}_j$ of $-\Delta_y$. For the FNLS case, we consider $\phi_j=\phi_j(y)$ then
\begin{equation}
    (-\Delta_x-\Delta_y)\phi_j=\lambda_j^2\phi_j,\quad \lambda_j>0.
\end{equation}
So we can write $u(x,y)$ by
\begin{equation}
    u(x,y)=\sum_{j}u_j(x)\phi_j(y),
\end{equation}
and $u_j(x)$ in Fourier satisfies
\begin{equation}
i\partial_t\hat{u}_j+ \parenthese{|\xi|^{2\sigma}+\lambda_j^{2\sigma} }\hat{u}_j=\hat{F}_j.    
\end{equation}
Hence we are reduced to the case of Guo-Wang \cite{GW} for the symbol $|\xi|^{2\sigma}+\lambda_j^{2\sigma}$ and the result follows. 
\end{proof}

For the record, we state below the  nonradial case, taking the $y$-direction into consideration:
\begin{lemma}
Consider \eqref{gNLShomog}. For $(\tilde{p},\tilde{q})$ and $(a,b)$ as in Lemma \ref{nonradialStri}, we have
\begin{equation}
   \|e^{it(-\Delta)^{\sigma}} u_0\|_{S^{s}_{p,q}H_y^{\gamma}} \lesssim \|u_0\|_{H_x^sH_y^{\gamma}}
\end{equation}
and
\begin{equation}
    \|\int_0^{t}e^{i(t-s)(-\Delta)^{\sigma}} F(s)ds\|_{S^{s}_{p,q}H_y^{\gamma}} \lesssim \||\nabla_x|^s F\|_{L_{t\in I}^{a'}L_x^{b'}H_y^{\gamma}}.
\end{equation}
\end{lemma}

At last, we recall the following useful lemma (see \cite{TV2})
\begin{lemma}\label{delta}
For every $0<s<1$, $p>0$ there exists $C = C(p, s)>0$ such that
\begin{equation}
    \|u|u|^p\|_{\dot{H}_y^s} \leq C \|u\|_{\dot{H}_y^s}\|u\|^p_{L_y^{\infty}}.
\end{equation}
\end{lemma}

\section{Well-posedness theory for $(gNLS)$ with vector-valued arguments}\label{sec LWP1}

In this section, we establish well-posedness theory for \eqref{gNLS} in Theorem \ref{mainthm} by the standard contraction mapping method together with the conservation law (from local to global). The main work is to construct suitable function spaces and to show the natural Duhamel mapping is a contraction mapping. It is tightly based on the Strichartz estimate for \eqref{gNLS} on $\mathbb{R}^d\times \mathbb{T}$  and careful choices of the exponents.  We refer to Section 4 of \cite{TV2} for the NLS analogue and Section 3 of \cite{YYZ2} (the fourth-order NLS case).

We also note that the analysis in this section covers the standard NLS case (when $\sigma=1$), which is consistent with \cite{TV2}. It is essential to assume $\sigma>\frac{1}{2}$ as we can see from the proof shortly.

First, we introduce the integral operator by Duhamel formula,
\begin{equation}
    \Phi_{u_0}(u)=e^{it\parenthese{(-\Delta_x)^{\sigma}+(-\partial_y^2)^\sigma }} u_0-i\int_0^{t}e^{i(t-s)\parenthese{(-\Delta_x)^{\sigma}+(-\partial_y^2)^\sigma }} (|u|^{p}u) \, ds.
\end{equation}
Then we can construct function space and show $\Phi_{u_0}$ is a contraction mapping. For the defocusing case, conservation law allows us to extend local well-posedness  to global well-posedness.

We define three norms,
\begin{equation}
    \|u\|_{X_T}=\|u\|_{L_t^qL_x^rH_y^{\frac{1}{2}+\delta}([-T,T]\times \mathbb{R}^d \times \mathbb{T})},
\end{equation}
\begin{equation}
    \|u\|_{Y_T^1}=\sum_{k=0,1}\| |\nabla_x|^{k\sigma} u\|_{L^l_tL^m_xL^2_y([-T,T]\times \mathbb{R}^d \times \mathbb{T})},
\end{equation}
and
\begin{equation}
    \|u\|_{Y_T^2}=\sum_{k=0,1}\| |\partial_y|^{k\sigma} u\|_{L^l_tL^m_xL^2_y([-T,T]\times \mathbb{R}^d \times \mathbb{T})},
\end{equation}
where $(l,m)$ is $\sigma$-admissible and $(q,r)$ is $\sigma$-admissible with $s$ regularity ($s+\frac{1}{2}+\delta \leq \sigma$). We will give the precise restrictions for the indices shortly, i.e. \eqref{res1}, \eqref{res2} and \label{res3}. Here we \textbf{note} that $\sigma>\frac{1}{2}$ such that it is ok to find a $s>0$. Combining the above three norms together, we define
\begin{equation}
    \|u\|_{Z_T}=\|u\|_{X_T}+\|u\|_{Y_T^1}+\|u\|_{Y_T^2}.
\end{equation}
Now we prove a contraction mapping for $\Phi_{u_0}$.

\textbf{Step 1. ($\Phi_{u_0}$ is from $Z_T$ to $Z_T$)}
Consider $X_T$ norm first. By Strichartz estimate, Lemma \ref{delta} and the H\"older,
\begin{equation}
    \aligned
    &\| |u|^p u \|_{L_t^{\tilde{q}'}L_x^{\tilde{r}'}H_y^{\frac{1}{2}+\delta}} \lesssim \| \|u\|^{p+1}_{H_y^{\frac{1}{2}+\delta}} \|_{L_t^{\tilde{q}'}L_x^{\tilde{r}'}}\\
    &\lesssim T^{\alpha(p)}\|u\|^{p+1}_{L_t^qL_x^rH_y^{\frac{1}{2}+\delta}}.
    \endaligned
\end{equation}
with $\alpha(p)>0$. Here we choose the indices such that,
\begin{equation}\label{res1}
    \frac{1}{\tilde{r}'}=\frac{p+1}{r},\frac{1}{\tilde{q}'}>\frac{p+1}{q}.
\end{equation}
It is manageable since the problem is subcritical, which is similar to the NLS case.

Then consider $Y_T^1$ and $Y_T^2$ norms. By Strichartz and the H\"older, for $k=0,1$,
\begin{equation}
    \aligned
    &\|D^k u|u|^p\|_{L^{l'}_tL^{m'}_xL^{2}_y}\lesssim \| \|D^k u\|_{L^{2}_y} \|u\|^p_{L_y^{\infty}}\|_{L^{l'}_tL^{m'}_x} \\
    &\lesssim \| \|D^k u\|_{L^{2}_y} \|u\|^p_{H_y^{\frac{1}{2}+\delta}}\|_{L^{l'}_tL^{m'}_x} \\
    &\lesssim T^{\alpha(p)}\|D^k u\|_{L^l_tL^m_xL^2_y}\|u\|^p_{L_t^qL_x^rH_y^{\frac{1}{2}+\delta}},
    \endaligned
\end{equation}
with $\alpha(p)>0$, where $D$ stands for $|\nabla_x|^{\sigma},|\partial_y|^{\sigma}$. (We \textbf{note} that we have used the fractional rule, i.e. Lemma A4 in Kato \cite{Kato}. See also Lemma 2.6 of Dinh \cite{Dinh}.)

Here we choose the indices such that,
\begin{equation}\label{res2}
    \frac{1}{m'}=\frac{1}{m}+\frac{p}{r},\quad \frac{1}{l'}>\frac{1}{l}+\frac{p}{q}.
\end{equation}
It is also manageable since the problem is subcritical, which is similar to the NLS case.

Thus, taking the above estimates into consideration, we can take the proper $T=T(\|u_0\|_{H^{\sigma}_{x,\alpha}})$ and 
$R=R(\|u_0\|_{H^{\sigma}_{x,\alpha}})$ such that $\Phi_{u_0}(B_{Z_{T'}}) \subset B_{Z_{T'}}$

\textbf{Step 2. ($\Phi_{u_0}$ is a contraction)} 
In this step, we show the contraction for $\Phi_{u_0}$. Let $T,R > 0$ be as in {\bf Step 1}. Then there exist $\bar{T}=\bar{T}(\|u_0\|_{
H^{\sigma}_{x,y}})<T$ such that $\Phi_{u_0}$ is a contraction on $B_{Z_{\bar{T}}}(0,R)$, equipped with the norm $L_{\bar{T}}^qL_x^rL^2_{y}$.

Using Strichartz estimate, the H\"older and the Sobolev inequality,
\begin{equation}
\aligned
    &\| \Phi_{u_0}(v_1)-\Phi_{u_0}(v_2)\|_{L_t^qL_x^rL^2_y}\lesssim \| v_1|v_1|^p- v_2|v_2|^p \|_{L_t^{\tilde{q}'}L_x^{\tilde{r}'}L^2_y}\\
    &\lesssim \| \|v_1-v_2\|_{L^2_y}(\|v_1\|^p_{L_y^{\infty}}+\|v_2\|^p_{L_y^{\infty}})\|_{L_t^{\tilde{q}'}L_x^{\tilde{r}'}} \\
    &\lesssim T^{\alpha(p)} \| v_1-v_2\|_{L_t^qL_x^rL^2_y}(\|v_1\|^{p}_{L_t^qL_x^rH_y^{\frac{1}{2}+\delta}}+\|v_2\|^{p}_{L_t^qL_x^rH_y^{\frac{1}{2}+\delta}}).
\endaligned    
\end{equation}
with $\alpha(p)>0$. Thus we conclude by taking $T$ small sufficiently.

\textbf{Step 3. (Uniqueness and Existence in $Z$)}

It is the same as the analogue in Section 4 in \cite{TV2}. We just use the contraction mapping argument so we skip it.

\textbf{Step 4. $u\in C((-T,T);H^{\sigma}_{x,y})$}

It is the same as the NLS analogue in Section 4 of \cite{TV2} so we omit it. We just use Strichartz estimates again as in Step 1 to guarantee that $u(t,x,y)\in \mathcal{C}((-T,T);H^{\sigma}_{x,y}) $.

\textbf{Step 5. (Unconditional uniqueness)}
We prove that for $u_1, u_2 \in C((-T, T);H^{\sigma}
_{x,y})$ are fixed points of $\Phi_{u_0}$, then $u_1 = u_2$.

Considering the difference of the integral equations satisfied by $u_1$ and $u_2$ and using Strichartz estimate,
\begin{align}
    \|u_1-u_2\|_{L^l_tL^m_xL^2_y} & \lesssim \|u_1|u_1|^p-u_2|u_2|^p\|_{L^{l'}_tL^{m'}_xL^2_y}  \\
    &\lesssim \|u_1-u_2\|_{L^l_tL^m_xL^2_y}(\|u_1\|^{p}_{L^{\frac{lp}{l-2}}L_x^{\frac{mp}{m-2}}L_y^{2}}+\|u_2\|^{p}_{L^{\frac{lp}{l-2}}L_x^{\frac{mp}{m-2}}L_y^{2}})\\
    &\lesssim \|u_1-u_2\|_{L^l_tL^m_xL^2_y}T^{\alpha(p)}(\|u_1\|^{p}_{L^{\infty}L_x^{\frac{mp}{m-2}}L_y^{2}}+\|u_2\|^{p}_{L^{\infty}L_x^{\frac{mp}{m-2}}L_y^{2}}).
\end{align}
with $\alpha(p)>0$. It is now like the NLS case. We can let $T$ be small enough to ensure uniqueness. (For Sobolev inequality reason) We note that 
\begin{equation}\label{rec3}
    2<\frac{mp}{m-2}<\frac{2d}{d+1-2\sigma},
\end{equation}
is required. We note again we still need $\sigma>\frac{1}{2}$ in this estimate.

The proof for the global well-posedness part in Theorem \ref{mainthm} is now complete using vector-valued arguments. 

\section{Well-posedness theory for $(gNLS)$ with decoupling arguments }\label{sec LWP2}

This section is devoted to the proof of the well-posedness aspects of Theorem \ref{mainthm} using decoupling arguments, in the spirit of the strategy designed by Schippa \cite{Schippa}. To be more specific, we consider $p=3$ in \eqref{gNLS} and assume for sake of generalization that the index of the regularity of the initial data $u_0$ is $s$. We then have
\begin{theorem}\label{main2}
There exists $s_{0}(d, \sigma)$ such that the initial value problem \eqref{gNLS} with $p=2$ is locally well-posed for $s> s_{0}(d, \sigma)$.
\end{theorem}

Since the dimension of the $y$ component here is easier to handle than in the previous section, we choose to consider directly the general waveguide $\mathbb{R}^{d-n} \times \mathbb{T}^n$. The argument is based on the approximation of the Euclidean component as in \cite{Barron} together with the decoupling theorem of Bourgain and Demeter \cite{BD}. We will even here consider the case of a multiplier $\mathfrak{m}(\xi,\eta)=|\xi|^{2\sigma}+|\eta|^2$, which means that the diffusion is Euclidean in tori directions. The adaptation to an anomalous diffusion is straightforward. 

On a dyadic interval $I_j=(2^{-j},2^{-j+1})$, the mutliplier $\mathfrak{m}$ behaves at {\sl low frequencies} like $|\xi|^{2\sigma}$ and at {\sl high frequencies} like $|\eta|^{2}$. An important tool in decoupling arguments is the behavior of the hessian of the phase function. The hessian $D^2\mathfrak{m}(\xi,\eta)$ is block-diagonal with a first $(d-n)\times (d-n)$-matrix corresponding to the hessian of the function $\xi \to |\xi|^{2\sigma}$ and a second block which is an $n\times n-$matrix which is a multiple of the identity. Therefore, the degeneracy of the phase $\mathfrak{m}$ is only dictated by the behavior of the eigenvalues at a given $\xi$ of the hessian of $|\xi|^{2\sigma}$ which are comparable on a dyadic interval to $2\sigma\,(2\sigma-1)\,N^{2(\sigma-1)}$ for a dyadic integer $N$.  Notice that the convexity of the phase changes according to $\sigma$ w.r.t. the value $\sigma=1/2$.  Finally notice that for $N$ dyadic $\nabla |\xi|^{2\sigma} \sim N^{2\sigma-1}$. We denote by $\psi(N):= N^{2(\sigma-1)}+N^2$ the phase function describing the behavior at dyadic scale $N$ of the eigenvalues of the Hessian of $\mathfrak{m}$. We also denote by $k$ the minimum between the number of negative eigenvalues and positive eigenvalues of $D^2\mathfrak{m}$, so that actually $k=0$ whenever $\sigma >1/2$. 

We start by recalling the following decoupling-type lemma. 
\begin{lemma}\label{decouple}
Suppose $g$, $g_l$ are Schwartz functions on $\mathbb{R}^{d-n}\times \mathbb{T}^n$ with $g=P_{\leq N}g$, such that
\begin{equation}
    g(x,y)=\sum_{l\in \mathbb{Z}^n,|l| \leq N}\int_{[-N,N]^{d-n}}\hat{g_l}(\xi) e^{2\pi i(x\cdot \xi+y\cdot l)} \, d\xi.
\end{equation}
Cover $[-N,N]^{d-n}$ by finitely-overlapping cubes $Q_k$ of side-length $\sim 1$, let $\{\phi_k\}$ be a partition of unity adapted to the $Q_k$, and define $g_{\theta_{m,k}}=e^{2\pi i y \cdot m}\mathcal{F}_x^{-1}(\hat{g}_m \phi_k)$. Then for $p \geq \frac{2(d+2-k)}{d-k}$ and any time interval $I$ of length $\sim 1$ we have
\begin{equation}\label{maineq3.1}
    \|e^{it\mathfrak{m}(\nabla/i)} g\|_{L^p(I \times \mathbb{R}^{d-n}\times \mathbb{T}^n)} \lesssim_{\epsilon}  \frac{N^{\epsilon+\frac{d}{2}-\frac{d+2}{p}}}{(\textmd{min}\{\mathfrak{m}(N),1\})^{\frac{1}{p}}}\bigg( \sum_{m,k} \|e^{\frac{it\mathfrak{m}(N\nabla/i)}{N^2 \mathfrak{m}(N)}} g_{\theta_{m,k}}  w_I\|^2_{L^p(\mathbb{R} \times \mathbb{R}^{d-n}\times \mathbb{T}^n)} \bigg)^{\frac{1}{2}},
\end{equation}
where $w_I$ is a bump function adapted to $I$.
\end{lemma}
\begin{proof}
The proof of Lemma \ref{decouple} is similar to the proof of the discrete restriction theorem in Bourgain-Demeter \cite{BD}. We approximate functions on product space $\mathbb{R}^{d-n}\times \mathbb{T}^n$ by functions on $\mathbb{R}^{d}$, apply the decoupling theorem, and then take limits. We refer to Lemma 3.2 of Barron \cite{Barron} for the standard Strichartz case, which has similar spirits. 
There are two differences from Barron \cite{Barron} that we need to be careful about. First, the range for $p$ is different (more narrow). Also, we need to make discussions regarding the value of $\mathfrak{m}(N)$, compared to $1$.
Let $B_l \in \mathbb{R}^{d-n}$ be a fixed ball of radius $N$ if $\mathfrak{m}(N)\ll 1$ (radius $N\mathfrak{m}(N)$ if $\mathfrak{m}(N) \gtrsim 1$), and let $w_l$ be a smooth weight adapted to $B_l \times [-1, 1]$. We observe that to prove Lemma \ref{decouple} it will suffice to show that,
\begin{equation}\label{bl}
     \|e^{it\mathfrak{m}(\nabla/i)} g\|_{L^p(I \times B_l \times \mathbb{T}^n)} \lesssim_{\epsilon}  \frac{N^{\epsilon+\frac{d}{2}-\frac{d+2}{p}}}{(\textmd{min}\{\mathfrak{m}(N),1\})^{\frac{1}{p}}}\bigg( \sum_{m,k} \|e^{\frac{it\mathfrak{m}(N\nabla/i)}{N^2 \mathfrak{m}(N)}} g_{\theta_{m,k}}  w_I\|^2_{L^p(w_l)} \bigg)^{\frac{1}{2}}.
\end{equation}
Then to prove the full estimate \eqref{maineq3.1} on $\mathbb{R}^{d-n}\times \mathbb{T}^n$, we choose a finitely-overlapping collection of balls $B_l$ that cover $\mathbb{R}^{d-n}$, and then apply \eqref{bl} in each $B_l$ and use Minkowski's inequality to sum that:
\begin{equation}
\aligned
   \|e^{it\mathfrak{m}(\nabla/i)} g\|^p_{L^p(I \times \mathbb{R}^{d-n} \times \mathbb{T}^n)} &\lesssim_{\epsilon} \sum_l    \|e^{it\mathfrak{m}(\nabla/i)} g\|^p_{L^p(I \times B_l \times \mathbb{T}^n)} \\ &\lesssim_{\epsilon} \frac{N^{\epsilon+\frac{d}{2}-\frac{d+2}{p}}}{(\textmd{min}\{\mathfrak{m}(N),1\})^{\frac{1}{p}}} \sum_l \bigg( \sum_{m,k} \|e^{\frac{it\mathfrak{m}(N\nabla/i)}{N^2 \mathfrak{m}(N)}} g_{\theta_{m,k}}  \|^2_{L^p(w_I)} \bigg)^{\frac{p}{2}} \\  
   &\lesssim_{\epsilon} \frac{N^{\epsilon+\frac{d}{2}-\frac{d+2}{p}}}{(\textmd{min}\{\mathfrak{m}(N),1\})^{\frac{1}{p}}} \bigg( \sum_{m,k} \|e^{\frac{it\mathfrak{m}(N\nabla/i)}{N^2 \mathfrak{m}(N)}} g_{\theta_{m,k}}  \|^2_{L^p(w)} \bigg)^{\frac{p}{2}},
\endaligned   
\end{equation}
which completes the proof of Proposition \ref{decouple}. Thus it suffices to prove \eqref{bl}. This reduction process is standard. We note that we have to discuss two situations: i.e. $\mathfrak{m}(N)\ll 1$ and $ \mathfrak{m}(N) \gtrsim 1 $ respectively. The two cases are essentially similar and the main difference is the integral range. Let's consider the first case as an example. For the other case, we can modify the arguments accordingly.\vspace{3mm}

We start with rescaling $u_0$ to have frequency support in $[-1,1]^d$. We let
\begin{align}
u(x,y,t)=e^{it\mathfrak{m}(\nabla/i)} g .
\end{align}
Note that
\begin{align}
u(N^{-1}x,N^{-1}y,N^{-2}\mathfrak{m}(N)^{-1}t) = N^{d-n}\int_{B_1^{d-n}} \sum_{m\in N^{-1}\mathbb{Z}^n \cap B_1^{n}} \hat{g}(N\xi,Nm) e^{i(x\cdot \xi+y\cdot m+t(\frac{\mathfrak{m}(N\xi)}{N^2 \mathfrak{m}(N)}+\frac{\mathfrak{m}(Nm)}{N^2 \mathfrak{m}(N)}))} \, d\xi.
\end{align}
For convenience, we denote the set $\Lambda_N=N^{-1}\mathbb{Z}^n \cap B_1^{n}$. Similarly to Barron \cite{Barron}, we then let $Ef$ denote the extension operator as follows,
\begin{equation}
    Ef=\int_{B_1^{d-n}} \sum_{m\in N^{-1}\mathbb{Z}^n \cap B_1^{n}} f(\xi,m) e^{i(x\cdot \xi+y\cdot m+t(\frac{\mathfrak{m}(N\xi)}{N^2 \mathfrak{m}(N)}+\frac{\mathfrak{m}(Nm)}{N^2 \mathfrak{m}(N)}))} \, d\xi,
\end{equation}
where $f(\xi,m)=\hat{g}(N\xi,Nm)$. After applying a change of variables on the spatial side and using periodicity in the $y$ variable, we see that
\begin{equation}
\aligned
    \|u\|_{L^p(B_N \times \mathbb{T}^n \times [0,1])}&=N^{d-n}N^{\frac{-(d+2)}{p}}\mathfrak{m}(N)^{-\frac{1}{p}}\|Ef\|_{L^p(B_{N^2} \times N\mathbb{T}^n \times [0,N^2\mathfrak{m}(N)])}\\
    &=N^{d-n}N^{\frac{-(d+2)}{p}}\mathfrak{m}(N)^{-\frac{1}{p}}N^{-\frac{n}{p}}\|Ef\|_{L^p(B_{N^2} \times N^2\mathbb{T}^n \times [0,N^2\mathfrak{m}(N)])} \\
    &\leq N^{d-n}N^{\frac{-(d+2)}{p}}\mathfrak{m}(N)^{-\frac{1}{p}}N^{-\frac{n}{p}}\|Ef\|_{L^p(B_{N^2} \times N^2\mathbb{T}^n \times [0,N^2])}.
    \endaligned
\end{equation}
Then following Barron \cite{Barron}, we introduce the operator $\tilde{E}$ defined by 
\begin{equation}
    \tilde{E}f=\int_{B_1^{d-n}} \int_{B_1^{n}} f(\xi_1,\xi_2) e^{i(x\cdot \xi_1+y\cdot \xi_2+t(\frac{\mathfrak{m}(N\xi_1)}{N^2 \mathfrak{m}(N)}+\frac{\mathfrak{m}(N\xi_2)}{N^2 \mathfrak{m}(N)}))} \, d\xi_1 d\xi_2.
\end{equation}
Given a function $f$ on $[-1,1]^{d-n} \times \Lambda_N$, let
\begin{equation}
    f^{\delta}(\xi_1,\xi_2)=\sum_{m\in \Lambda_N}c_d \delta^{-d}1_{\{|\xi_2-m|\leq \delta\}}f(\xi_1,m),\quad \delta< \frac{1}{N},
\end{equation}
where $c_d$ is a dimensional constant chosen for normalization. Then by Lebesgue differentiation and Fatou lemma, we have
\begin{equation}\label{limiting}
   \|Ef\|_{L^p(B_{N^2} \times N^2\mathbb{T}^n \times [0,N^2])} \leq \liminf_{\delta \rightarrow 0} \|\tilde{E}f^{\delta}\|_{L^p(B_{N^2} \times N^2\mathbb{T}^n \times [0,N^2])}.
\end{equation}
We will begin by estimating $\tilde{E}f$ for arbitrary $f$
on $B_1^d$ before specializing to $f^{\delta}$ and passing to the limit later in the argument. The above process reduces the waveguide case to the Euclidean case. Then we can use Bourgain-Demeter's decoupling result \cite{BD} and the limiting argument \eqref{limiting} to obtain,
\begin{equation}
    \|Ef\|_{L^{p^{\ast}}(B_{N^2} \times N^2\mathbb{T}^n \times [0,N^2])} \leq \liminf_{\delta \rightarrow 0} N^{\epsilon+\alpha_d}\big( \sum_{m,k} \|\tilde{E}f^{\delta}_{m,k}\|^2_{L^{p^{\ast}}(w_{N^2})}\big).
\end{equation}
The rest of the proof follows as in Barron \cite{Barron}.
\end{proof}

We prove now the crucial (localized) Strichartz estimates whose proof has a similar spirit with Proposition 3.4 \cite{Barron}. The tori analogue of Lemma \ref{Strichartzgeneral} is proved in Schippa \cite{Schippa}.

 \begin{lemma}\label{Strichartzgeneral}
Let the  interval $I$ be compact. Then, we have the following estimates, holding up to any $\epsilon>0$,
\begin{equation}\label{mainStri}
    \|P_N e^{it\mathfrak{m}(\nabla/i)} f\|_{L^p(I\times \mathbb{R}^{d-n} \times \mathbb{T}^n)}\lesssim_{\epsilon,|I|} \frac{N^{\frac{d}{2}-\frac{d+2}{p}+\epsilon}}{(\textmd{min}\{\mathfrak{m}(N),1\})^{\frac{1}{p}}}\|f\|_{L^2(\mathbb{R}^{d-n} \times \mathbb{T}^n)},
\end{equation}
where $p\geq \frac{2(d+2-k)}{d-k}$. The number $k$ has been introduced before. 
\end{lemma}

\begin{proof}
Suppose $f=P_{\leq N}f$. By interpolating with $p=\infty$ (via Bernstein’s inequality) it suffices to prove the endpoint case $p= \frac{2(d+2-k)}{d-k}$. By Lemma \ref{decouple},
\begin{equation}
    \|e^{it\mathfrak{m}(\nabla/i)} f\|_{L^p(I \times \mathbb{R}^{d-n}\times \mathbb{T}^n)} \lesssim_{\epsilon}  \frac{N^{\epsilon+\frac{d}{2}-\frac{d+2}{p}}}{(\textmd{min}\{\mathfrak{m}(N),1\})^{\frac{1}{p}}}\bigg( \sum_{m,k} \|e^{\frac{it\mathfrak{m}(N\nabla/i)}{N^2 \mathfrak{m}(N)}} f_{\theta_{m,k}}  w_I\|^2_{L^p(\mathbb{R} \times \mathbb{R}^{d-n}\times \mathbb{T}^n)} \bigg)^{\frac{1}{2}}.
\end{equation}
By Plancherel theorem it suffices to prove the desired estimate when $f = P_{\theta}f$ and $\theta=\theta_{m,k}$.
In this case, we apply H\"older's inequality in time to get
\begin{equation}
    \|e^{\frac{it\mathfrak{m}(N\nabla/i)}{N^2 \mathfrak{m}(N)}} f_{\theta}  w_I\|_{L^p(\mathbb{R} \times \mathbb{R}^{d-n}\times \mathbb{T}^n)} \lesssim \|e^{\frac{it\mathfrak{m}(N\nabla/i)}{N^2 \mathfrak{m}(N)}} f_{\theta} \|_{L^q_tL_x^p(\mathbb{R} \times \mathbb{R}^{d-n})}
\end{equation}
where $q=\frac{4p}{(d-n)(p-2)}$ is the admissible time exponent for the Strichartz estimate on $\mathbb{R}^{d-n}$. Applying the Strichartz estimate for fractional Schr\"odinger operator (see \cite{GW,CHKL,Dinh} and the reference therein), 
\begin{equation}
  \|e^{\frac{it\mathfrak{m}(N\nabla/i)}{N^2 \mathfrak{m}(N)}} f_{\theta} \|_{L^q_tL_x^p(\mathbb{R} \times \mathbb{R}^{d-n})} \lesssim \| f_{\theta} \|_{L_x^2(\mathbb{R} \times \mathbb{R}^{d-n})}.
\end{equation}
This completes the proof.\vspace{5mm}

\end{proof}

\begin{remark}
By interpolation with the trivial bound $L^{\infty}_tL_x^2$ or using the Bernstein inequality, a $L_t^pL_x^q$-version Strichartz estimate can be obtained as below 
\begin{equation}\label{mainStri2}
    \|P_N e^{it\mathfrak{m}(\nabla/i)} u_0\|_{L_t^pL^q_x(I\times \mathbb{R}^{d-n} \times \mathbb{T}^n)}\lesssim_{\epsilon,|I|} \frac{N^{\frac{d}{2}-\frac{2}{p}-\frac{d}{q}+\epsilon}}{(\textmd{min}\{\mathfrak{m}(N),1\})^{\frac{1}{p}}},
\end{equation}
where $p,q \geq \frac{2(d+2-k)}{d-k}$. 
\end{remark}

One can then prove the following bilinear estimates which give the desired  well-posedness result for the cubic generalized NLS model in the setting of waveguide manifolds, i.e. the statement in Theorem \ref{main2}.

\begin{proposition}\label{bi1}
 Let $I$ be a compact interval. Then, there exists $s(d,k)$ such that we have the estimate,
\begin{equation}
   \big\|P_Ne^{it\mathfrak{m}(\nabla/i)}u_0 P_Ke^{it\mathfrak{m}(\nabla/i)}v_0\big\|_{L^2_{t,x}(I\times \mathbb{R}^{d-n}\times \mathbb{T}^n)} \lesssim_{C_s,|I|}K^{2s} \|P_N u_0\|_{L^2}\|P_K u_0\|_{L^2}  
\end{equation}
to hold for $s>s(d,k)$.
\end{proposition}
\begin{proof}
We note that the proof is based on the Strichartz Lemma \ref{Strichartzgeneral}. We refer to Proposition 1.3 of \cite{Schippa} for the tori version. 

Let $P_N =\sum_{K_1}R_{K_1}$, where $R_K$ projects to cubes of side-length $K$. Then, by means of almost orthogonality
\begin{equation}
    \big\|P_Ne^{it\mathfrak{m}(\nabla/i)}u_0 P_Ke^{it\mathfrak{m}(\nabla/i)}v_0\big\|^{2}_{L^2_{t,x}(I\times \mathbb{R}^{d-n}\times \mathbb{T}^n)} \lesssim \sum_{K_1} \big\|P_{K_1}e^{it\mathfrak{m}(\nabla/i)}u_0 P_{K}e^{it\mathfrak{m}(\nabla/i)}v_0\big\|^{2}_{L^2_{t,x}(I\times \mathbb{R}^{d-n}\times \mathbb{T}^n)}
\end{equation}
In viewing of H\"older’s inequality we are left with estimating two $L^{4}_{t,x}$-norms.
Clearly, by Lemma \ref{Strichartzgeneral},
\begin{equation}
\big\|P_Ke^{it\phi(\nabla/i)}v_0\big\|_{L^4_{t,x}(I\times \mathbb{R}^{d-n}\times \mathbb{T}^n)}\lesssim K^s \|P_K v_0\|_{L^2}.    
\end{equation}
Then it suffices to treat the other term. The rest of the proof follows line to line from Proposition 1.3 of \cite{Schippa}  so we omit. 
\end{proof}

\section{Morawetz estimates on waveguides}\label{sec Morawetz}

In this section, we establish a Morawetz estimate for solutions to \eqref{gNLS} on $\R^d \times \T$. This step is crucial to obtain the decay property for solutions of \eqref{gNLS}. 

We first define the following Morawetz action on the waveguide $\R^d \times \T$:
\begin{align}
M_{\phi} [u(t)] := 2 \im \int_{\R^d \times \T} \overline{u} (t,x ,y) \nabla_x \phi (x) \cdot \nabla_x  u(t,x,y) \, dxdy .
\end{align}
Note that, employing the similar idea in \cite{PV},  the weight function $\nabla_x \phi (x)$ that we chose here depends on only $x$, and not on $y$.

Then we present the main result in this section.  For a ready-to-use Morawetz estimate, see Corollary \ref{cor Morawetz}. 
\begin{lemma}\label{lem Morawetz}
If $u$ solves \eqref{gNLS}, then the Morawetz action satisfies the identity
\begin{align}
\frac{d}{dt}  M_{\phi} [u(t)]  = \int_{0}^{\infty} m^s \int_{\R^d \times \T}  \parenthese{ 4 \overline{ \partial_{x_k} u_m } (\partial_{x_k x_l} \phi) \partial_{x_l} u_m - \Delta_x^2 \phi \abs{u_m}^2 } \, dx d y dm - \frac{2p \mu}{p+2} \int_{\R^d \times \T} \Delta_x \phi \abs{u}^{p+1} \, dx d y .
\end{align}
\end{lemma}

\begin{proof}[Proof of Lemma \ref{lem Morawetz}]
Following the strategy in \cite{BHL}, we define
\begin{align}
\Gamma_{\phi} : =   i (\nabla_x \cdot \nabla_x \phi + \nabla_x \phi \cdot \nabla_x) , 
\end{align}
that is
\begin{align}
\Gamma_{\phi} f : = i [\nabla_x \cdot ((\nabla_x \phi) f) + \nabla_x \phi \cdot \nabla_x f] .
\end{align}

Under this notation, we claim that 
\begin{align}\label{eq claim}
\inner{u(t) , \Gamma_{\phi} (t)} & =  - M_{\phi} [u(t)] . 
\end{align}
Note that $\inner{f,g} = \re \int_{\R^d \times \T} \overline{f} g \, dx dy$.

In fact,
\begin{align}
\inner{u(t) , \Gamma_{\phi} (t)} & = \inner{u(t) ,  i [\nabla_x \cdot ((\nabla_x \phi) u) + \nabla_x \phi \cdot \nabla_x u] }  = \inner{u , i \nabla_x \cdot ( (\nabla_x \phi) u)} + \inner{u, i \nabla_x \phi \cdot \nabla_x u} .
\end{align}
We then compute the two inner products separately: 
\begin{align}
\inner{u ,  i \nabla_x \phi \cdot \nabla_x u} &  = \re \int_{\R^d \times \T} \overline{u} ( i \nabla_x \phi \cdot \nabla_x u) \, dx d y  =  - \im \int_{\R^d \times \T} \overline{u} \nabla_x \phi \cdot \nabla_x u \, dx dy  =  - \frac{1}{2} M_{\phi} ;
\end{align}
and
\begin{align}
\inner{u ,  i \nabla_x \cdot ( (\nabla_x \phi) u)} & = \re \int_{\R^d \times \T} \overline{u} i \nabla_x \cdot ((\nabla_x \phi) u) \, dx dy  =  - \im \int_{\R^d \times \T}  \overline{u} \partial_{x_l} (\partial_{x_l} \phi u) \, dx dy \\
& =  - \im  \int_{\R^d \times \T} \overline{u} \partial_{x_l x_l} \phi u \, dx dy   - \im \int_{\R^d \times \T}  \overline{u} \partial_{x_l} \phi \partial_{x_l} u \, dx dy \\
& =  - \im \int_{\R^d \times \T} \overline{u} \nabla_x \phi \cdot \nabla_x u \, dx dy  =  - \frac{1}{2} M_{\phi} .
\end{align}

Therefore, by combining these two terms, we conclude the claim \eqref{eq claim}
\begin{align}
\inner{u(t) , \Gamma_{\phi} (t)} & =  -  M_{\phi} [u(t)] .
\end{align}

Next, we compute the derivative of $M_{\phi} [u(t)] $ with respect to time $t$. Using \eqref{gNLS}
\begin{align}
\partial_t u  = i\parenthese{(-\Delta_x)^\sigma+(-\partial^2_y)^\sigma } u -i\mu|u|^pu
\end{align}

and  Plancherel theorem, we write
\begin{align}
\frac{d}{dt} M_{\phi} [u(t)] & = \inner{\frac{d}{dt} u(t), \Gamma_{\phi} u(t)} + \inner{u(t),  \frac{d}{dt} \Gamma_{\phi} u(t)} \\
& = \inner{i  \parenthese{(-\Delta_x)^\sigma+(-\partial^2_y)^\sigma }  u - i\mu  \abs{u}^p u , \Gamma_{\phi} u(t)} + \inner{u(t) , \Gamma_{\phi} \frac{d}{dt}u(t) } \\
& = \inner{i \parenthese{(-\Delta_x)^\sigma+(-\partial^2_y)^\sigma } u, \Gamma_{\phi} u(t)} - \inner{ i \mu \abs{u}^p u , \Gamma_{\phi} u(t)} \\
& \quad + \inner{u(t) , i \Gamma_{\phi} ( \parenthese{(-\Delta_x)^\sigma+(-\partial^2_y)^\sigma } u -\mu \abs{u}^p u )} \\
& = - \inner{u(t) , \parenthese{(-\Delta_x)^\sigma+(-\partial^2_y)^\sigma } i \Gamma_{\phi} u(t)} + \inner{u(t) , i \Gamma_{\phi} \parenthese{(-\Delta_x)^\sigma +(-\partial^2_y)^\sigma } u(t)}  \\
& \quad + \inner{u(t) , \mu \abs{u}^p i \Gamma_{\phi} u(t)} - \inner{u(t), i \Gamma_{\phi} (\mu \abs{u}^p u)} \\
& = -\inner{u(t), [(-\Delta_x)^\sigma+(-\partial^2_y)^\sigma  , i \Gamma_{\phi}] u(t)} + \inner{u(t) , [\mu\abs{u}^p, i \Gamma_{\phi}] u(t)}  ,
\end{align}
where we used the commutator  notation $[A,B] =AB - BA$. 

Noticing that 
\begin{align}
[(-\partial^2_y)^\sigma  , i \Gamma_{\phi}]  =0 ,
\end{align}
we then have
\begin{align}
\frac{d}{dt} M_{\phi} [u(t)] & =  \underbrace{-  \inner{u(t), [(-\Delta_x)^\sigma , i \Gamma_{\phi}] u(t)}}_{I} + \underbrace{ \inner{u(t) , [\mu \abs{u}^p, i \Gamma_{\phi}] u(t)}}_{II}  = :  I + II .
\end{align}


In the rest of the proof, we will work on the linear term $I$ and the nonlinear term $II$ separately.

First, we consider the linear term $I$. In order to deal with the $[(-\Delta_x)^\sigma , i \Gamma_{\phi}]$ term inside $I$, we will employ the following Balakrishinan’s representation formula for $(-\Delta_x)^{\sigma}$ introduced in \cite{Balakrishnan} for $\sigma \in (0,1)$, 
\begin{align}\label{eq Bala}
(-\Delta_x)^{\sigma} = \frac{\sin (\pi \sigma)}{\pi} \int_0^{\infty} m^{\sigma -1}  \frac{-\Delta_x}{-\Delta_x + m} \, dm .
\end{align}

In general, for $A \geq 0$, $m >0$, the following commutator has the form of 
\begin{align}
[\frac{A}{A + m} , B] =  m\frac{1}{A +m} [A,B] \frac{1}{A+m} .
\end{align}
In particular, if taking $A = -\Delta_x$ and combining with \eqref{eq Bala}, we write
\begin{align}\label{eq 6}
[(-\Delta_x)^{\sigma} , B] = \frac{\sin (\pi \sigma)}{\pi} \int_0^{\infty} m^{\sigma} \frac{1}{-\Delta_x + m} [-\Delta_x , B] \frac{1}{-\Delta_x + m} \, dm .
\end{align}

Then taking $B = i\Gamma_{\phi}$ in \eqref{eq 6}, we have 
\begin{align}\label{eq 7}
[(-\Delta_x)^{\sigma} , i\Gamma_{\phi}] = \frac{\sin (\pi \sigma)}{\pi} \int_0^{\infty} m^{\sigma} \frac{1}{-\Delta_x + m} [-\Delta_x , i\Gamma_{\phi}] \frac{1}{-\Delta_x + m} \, dm .
\end{align}

Now we claim that 
\begin{align}\label{eq 9}
[-\Delta_{x} , i \Gamma_{\phi}] = 4 \partial_{x_k } (\partial_{x_k x_l} \phi) \partial_{x_l} + \Delta_x^2 \phi , 
\end{align}
where to emphasize that $\Delta_x$ takes derivative only in the $\R^d$ direction, we put $x$ in its subscript. Similarly, in the following calculations, $\partial_{x_l}$ and $\partial_{x_k}$ are differential operators in $\R^d$ directions, while $\partial_y$ is the $\T$ direction derivative.

In fact, recalling
\begin{align}\label{eq 4}
i \Gamma_{\phi} f : =  -[\nabla_x \cdot ((\nabla_x \phi) f) + \nabla_x \phi \cdot \nabla_x f] ,
\end{align}
we write
\begin{align}
[-\Delta_{x} , i \Gamma_{\phi}] f & = -\Delta_{x} (i\Gamma_{\phi}) f + i \Gamma_{\phi} (\Delta_{x} f) \\
& =  \Delta_{x} \nabla_x \cdot ( (\nabla_x \phi)f) + \Delta_{x} (\nabla_x \phi \cdot \nabla_x f) - \nabla_x ((\nabla_x \phi) \Delta_{x} f ) - \nabla_x f \cdot \nabla_x ( \Delta_{x} f) \\
& =  \partial_{x_k x_k} \partial_{x_l} (\partial_{x_l} \phi  f) + \partial_{x_k x_k} (\partial_{x_l} \phi \partial_{x_l} f) - \partial_{x_l} (\partial_{x_l} \phi \partial_{x_k x_k} f) - \partial_{x_l} \partial_{x_l x_k x_k} f  . \label{eq 1} 
\end{align}

Then using the product rule, we  continue from \eqref{eq 1} 
\begin{align}
& \quad [-\Delta_{x} , i \Gamma_{\phi}] f  \\
& =  \partial_{x_k x_k} (\partial_{x_l x_l} \phi f) + 2 \partial_{x_k x_k} (\partial_{x_l} \phi \partial_{x_l} f) - \partial_{x_l x_l} \phi \partial_{x_k x_k} f - \partial_{x_l} \phi \partial_{x_k x_k x_l} f - \partial_{x_l} f \partial_{x_k x_k x_l} f\\
& = \partial_{x_k x_k x_l x_l} \phi f + 2 \partial_{x_k x_l x_l} \phi \partial_{x_k} f + \partial_{x_l x_l} \phi \partial_{x_k x_k} f + 2 \partial_{x_k x_k x_l} \phi \partial_{x_l} f + 4 \partial_{x_k x_l} \phi \partial_{x_k x_l} f + 2 \partial_{x_l} \phi \partial_{x_k x_k x_l} f \\
& \quad - \partial_{x_l x_l} \phi \partial_{x_k x_k} f - \partial_{x_l} \partial_{x_k x_k x_l} f - \partial_{x_l} \phi \partial_{x_k x_k x_l}f \\
& =  \Delta_x^2 \phi f + 4 \partial_{x_k x_l x_l} \phi \partial_{x_k} f + 4 \partial_{x_k x_l} \phi \partial_{x_k x_l} f\\
& =  4 \partial_{x_k } (\partial_{x_k x_l} \phi) \partial_{x_l} f + \Delta_x^2 \phi f . \label{eq 8}
\end{align}
This proves the claim \eqref{eq 9}.

At this point, we are in  a good position to compute the term $I$. 
First, combining \eqref{eq 9} and \eqref{eq 7}, we write
\begin{align}
[(-\Delta_x)^{\sigma}, i \Gamma_{\phi}] & = \frac{\sin (\pi \sigma)}{\pi} \int_0^{\infty} m^{\sigma} \frac{1}{-\Delta_x + m} [-\Delta_x , i \Gamma_{\phi}] \frac{1}{-\Delta_x +m}\, dm \\
& = \frac{\sin (\pi \sigma)}{\pi} \int_0^{\infty} m^{\sigma} \frac{1}{-\Delta_x + m} [4 \partial_{x_k } (\partial_{x_k x_l} \phi) \partial_{x_l}  + \Delta_x^2 \phi ] \frac{1}{-\Delta_x +m}\, dm  . \label{eq 5}
\end{align}
Therefore
\begin{align}
I = - \inner{u(t) , [(-\Delta_x)^{\sigma}, i \Gamma_{\phi}]u} & =  \int_{\R^d \times \T} \overline{u} \frac{\sin (\pi \sigma)}{\pi} \int_0^{\infty} m^{\sigma} \frac{1}{-\Delta_x + m} [-4 \partial_{x_k } (\partial_{x_k x_l} \phi) \partial_{x_l}  - \Delta_x^2 \phi ] \frac{1}{-\Delta_x +m} u\, dm dxdy .
\end{align}

For $m > 0$, we define 
\begin{align}
u_m (t) := c_{\sigma} \frac{1}{-\Delta_{x } +m} u(t) = c_{\sigma} \mathcal{F}^{-1} \parenthese{ \frac{\widehat{u}(t, \xi)}{\abs{\xi}^2 +m} } ,
\end{align}
where
\begin{align}
c_{\sigma} : = \sqrt{\frac{\sin (\pi \sigma)}{\pi}} .
\end{align}

Under such change of variables and with  Fubini's theorem, Plancherel theorem, and integration by parts, we obtain
\begin{align}
I & = \int_0^{\infty} \int_{\R^d \times \T} \overline{u} c_{\sigma} m^{\sigma} \frac{1}{-\Delta_x +m} [-4 \partial_{x_k } (\partial_{x_k x_l} \phi) \partial_{x_l}  - \Delta_x^2 \phi ]  u_m \, dx dy dm \\
& = \int_0^{\infty} m^{\sigma}  \int_{\R^d \times \T} \frac{c_{\sigma}}{-\Delta_x +m} \overline{u}   [-4 \partial_{x_k } (\partial_{x_k x_l} \phi) \partial_{x_l}  - \Delta_x^2 \phi ]  u_m \, dx dy dm \\
& =  \int_0^{\infty} m^{\sigma}  \int_{\R^d \times \T}  \overline{u_m}   [-4 \partial_{x_k } (\partial_{x_k x_l} \phi) \partial_{x_l}  - \Delta_x^2 \phi ]  u_m \, dx dy dm \\
& = \int_0^{\infty} m^{\sigma}  \int_{\R^d \times \T} 4 \overline{ \partial_{x_k} u_m } (\partial_{x_k x_l} \phi) \partial_{x_l} u_m - \Delta_x^2 \phi \abs{u_m}^2 \, dxdydm .
\end{align}

Then we consider  the nonlinear term $II$. Noticing that 
\begin{align}
\nabla (\abs{u}^{p+2}) = \frac{p+2}{p} \nabla (\abs{u}^p) \abs{u}^2 ,
\end{align}
and using \eqref{eq 4} again, we obtain 
\begin{align}
II = \inner{u(t), [\mu \abs{u}^p , i \Gamma_{\phi}] u(t)} & = - \mu \inner{u(t) , [\abs{u}^p , \nabla_x \cdot \nabla_x \phi + \nabla_x \phi \cdot \nabla_x] u} \\
& = - \mu \inner{u(t) , \abs{u}^p \nabla_x \cdot ((\nabla_x \phi)u)} - \mu \inner{u(t) , \abs{u}^p \nabla_x \phi \cdot \nabla_x u} \\
& = - \frac{2p \mu}{p+2} \int_{\R^d \times \T} (\Delta_x \phi) \abs{u}^{p+2} \, dx dy .
\end{align}

Therefore, together with the computation on terms $I$ and $II$, we get
\begin{align}
\frac{d}{dt}  M_{\phi} [u(t)] =  \int_{0}^{\infty} m^s \int_{\R^d \times \T} \parenthese{ 4 \overline{ \partial_{x_k} u_m } (\partial_{x_k x_l} \phi) \partial_{x_l} u_m - \Delta_x^2 \phi \abs{u_m}^2 } \, dx d y dm - \frac{2p\mu}{p+2} \int_{\R^d \times \T} \Delta_x \phi \abs{u}^{p+2} \, dx d y ,
\end{align}
which implies Lemma \ref{lem Morawetz}. 
\end{proof}

\begin{corollary}\label{cor Morawetz}
Assume  $u$ to be a smooth solution to the initial value problem \eqref{gNLS} with $d \geq 3$, then we have the following Morawetz inequality
\begin{align}
\int_{\R} \int_{\R^d \times \T} \frac{\abs{u(t,x,y)}^{p+2} }{\abs{x}} \, dx d y dt \lesssim  \sup_{t \in \R} \norm{u(t)}_{\dot{H}^{\frac{1}{2}}}^2 \lesssim  \sup_{t \in \R} \norm{u(t)}_{H^{\sigma}}^2 .
\end{align}
\end{corollary}

\begin{proof}[Proof of Corollary \ref{cor Morawetz}]
We take $\phi(x) = \abs{x}$ (independent on $y$) in Lemma \ref{lem Morawetz}. Hence
\begin{align}
\nabla_x \phi & = \frac{x}{\abs{x}} ,\\
\Delta_x \phi & = \frac{d-1}{\abs{x}} ,\\
\partial_{x_k x_l} \phi & = \frac{\delta_{x_k x_l}}{\abs{x}} - \frac{x_k x_l}{\abs{x}^3}\\
\Delta_x^2 \phi & = 
\begin{cases}
- \pi \delta(x) , & d=3,\\
- (d-1)(d-3)\abs{x}^{-3}, & d \geq 4 .
\end{cases}
\end{align}
Under such choice of $\phi $, we claim that
\begin{align}\label{eq 2}
\int_{0}^{\infty} m^s \int_{\R^d \times \T}  \parenthese{4 \overline{ \partial_{x_k} u_m } (\partial_{x_k x_l} \phi) \partial_{x_l} u_m - \Delta_x^2 \phi \abs{u_m}^2 } \, dx d y dm  \geq 0 .
\end{align}
Assuming \eqref{eq 2}, we obtain for the defocusing equation ($\mu=-1$ in \eqref{gNLS})
\begin{align}
\frac{d}{dt}   M_{\abs{x}} [u(t)] &  \geq \frac{2p}{p+2} \int_{\R^d \times \T} \Delta_x (\abs{x}) \abs{u(t,x,y)}^{p+2} \, dx d y \\
& =  \frac{2p}{p+2} \int_{\R^d \times \T} \frac{d-1}{\abs{x}} \abs{u(t,x,y)}^{p+2} \, dx d y , \label{eq 10}
\end{align}
which gives Corollary \ref{cor Morawetz} by combining the following upper bound of $M_{\abs{x}}$ in \cite{CKSTT}
\begin{align}
\abs{M_{\abs{x}} [u(t)]} \lesssim  \sup_{t } \norm{u(t)}_{\dot{H}^{\frac{1}{2}}}^2 \lesssim  \sup_{t} \norm{u(t)}_{H^{\sigma}}^2 ,
\end{align}
and integrating in $t$ using the fundamental theorem of calculus.

To see \eqref{eq 2}, we write (when $d \geq 4$)
\begin{align}
& \quad \int_{\R^d} \parenthese{4 \overline{ \partial_{x_k} u_m } (\partial_{x_k x_l} \phi) \partial_{x_l} u_m - \Delta_x^2 \phi \abs{u_m}^2 } \, dx \\
& = \int_{\R^d} \parenthese{ 4 \overline{ \partial_{x_k} u_m } (\frac{\delta_{x_k x_l}}{\abs{x}} - \frac{x_k x_l}{\abs{x}^3}) \partial_{x_l} u_m + \frac{(d-1)(d-3)}{\abs{x}^3} \abs{u_m}^2 }\, dx  .\label{eq 3}
\end{align} 
The $d =3$ case can be handled similarly, hence omitted.

Using the notation $\nabla_{\vec{e}} u = (\vec{e} \cdot u) \frac{\vec{e}}{\abs{\vec{e}}^2}$ and $\nabla_{\vec{e}}^{\perp} u = \nabla u - \nabla_{\vec{e}} u$ with $\vec{e} = \vec{x}$, we can decompose $\nabla u$ orthogonally. Then we have 
\begin{align}
 \overline{ \partial_{x_k} u_m } x_k x_l \partial_{x_l} u_m & \leq \frac{1}{2} \abs{\partial_{x_k} u_m x_k}^2 + \frac{1}{2} \abs{\partial_{x_l} u_m x_l}^2 \\
& \leq \frac{1}{2} \abs{\nabla_{\vec{e}} u_m}^2 \abs{x}^2 + \frac{1}{2} \abs{\nabla_{\vec{e}} u_m}^2 \abs{x}^2 \\
& = \abs{\nabla_x u_m}^2 \abs{x}^2 .
\end{align}
Then continuing from \eqref{eq 3}, we obtain
\begin{align}
\eqref{eq 3} & \geq \int_{\R^d} \parenthese{4 \frac{\abs{\nabla_x u_m}^2}{\abs{x}} - 4 \frac{\abs{\nabla_{\vec{e}} u_m}^2 \abs{x}^2}{\abs{x}^3} + \frac{(d-1) (d-3)}{\abs{x}^3} \abs{u_m}} \, dx \\
& = \int_{\R^d} \parenthese{4 \frac{\abs{\nabla_x u_m}^2 - \abs{\nabla_{\vec{e}} u_m}^2 }{\abs{x}} + \frac{(d-1)(d-3)}{\abs{x}^3} \abs{u_m}^2 } \, dx \\
& = \int_{\R^d} \parenthese{4 \frac{ \abs{\nabla_{\vec{e}}^{\perp} u_m}^2 }{\abs{x}} + \frac{(d-1)(d-3)}{\abs{x}^3} \abs{u_m}^2 } \, dx  \geq 0 , 
\end{align}
then  \eqref{eq 2} follows by integrating \eqref{eq 10} in both $y $ and $m$. This completes the proof of Corollary \ref{cor Morawetz}.

\end{proof}

\section{Proof of the scattering result}\label{sec Scattering}
In this section, we give the proof for the scattering result in Theorem \ref{mainthm}. There are four steps and we will discuss them step by step. The strategy has similar spirit with \cite{TV2, YYZ2}.
\subsection{Step 1: the Morawetz bound}
Recall the Morawetz estimate (Corollary \ref{cor Morawetz}) established in Section \ref{sec Morawetz}, we have 
\begin{equation}\label{Morawetz}
   \int_{\mathbb{R}} \int_{\mathbb{R}^d \times \mathbb{T}} \frac{|u(t,x,y)|^{2+p}}{|x|} \, dt dxdy \lesssim_{\textmd{data}} 1.
\end{equation}
\subsection{Step 2: Proof of the decay property}
Based on the Morawetz bound above, we aim to show the decay property of \eqref{gNLS}, i.e.
\begin{equation}
 \lim_{t \rightarrow \infty} \|u(t,x,y)\|_{L^q_{x,y} (\R^d \times \T)}=0,   
\end{equation}
where $2<q \leq 2+r$ (for any $r<\frac{2+4\sigma+dp+p+2p\sigma}{2d}$). This decay property is essential for us to obtain the scattering result. 
\begin{remark}
One may also consider the (stronger) pointwise type decay which describes the decay rate of nonlinear solutions quantitatively. See \cite{fan2021decay,yu2022decay} for recent results and the references therein.
\end{remark}
In viewing of interpolation with the mass conservation law, it suffices to show the endpoint case, that is,
\begin{equation}\label{decay}
 \lim_{t \rightarrow \infty} \|u(t,x,y)\|_{L^{2+r}_{x,y} (\R^d \times \T)}=0.   
\end{equation}
We will prove it by contradiction. Before starting with the proof, we recall a radial Sobolev embedding as follows,
\begin{lemma}[Radial Sobolev Embeddings in $\R^d$ in \cite{TVZ}]\label{prop Radial Sobolev}
Let $d \geq 1$, $1 \leq q \leq \infty$, $0 < s < d$ and $\beta \in \R$ obey the conditions
\begin{align*}
\beta > -\frac{d}{q}, \quad 0 \leq \frac{1}{p} -\frac{1}{q} \leq s
\end{align*}
and the scaling condition
\begin{align*}
\beta + s = \frac{d}{p} - \frac{d}{q}
\end{align*}
with at most one of the equalities
\begin{align*}
p=1, \quad p=\infty, \quad q=1, \quad q=\infty, \quad \frac{1}{p} - \frac{1}{q} =s
\end{align*}
holding. Then for any spherically symmetric function $f \in \dot{W}^{s,p} (\R^d)$, we have 
\begin{align*}
\norm{\abs{x}^{\beta} f}_{L^q (\R^d)} \lesssim \norm{\abs{\nabla}^s f}_{L^p (\R^d)}.
\end{align*}
\end{lemma}

Let $\beta$ satisfies $(2+p)\beta+1=2+r$.

Via Lemma \ref{prop Radial Sobolev},  H\"older inequality and Sobolev embedding, we have
\begin{align}
\|u(t,x,y)\|_{L^{2+r}_{x,y} (\R^d \times \T)}  &= \big(\int |u(t,x,y)|^{2+r} \, dx dy \big)^{\frac{1}{2+r}}\\  &=\big(\int \frac{|u(t,x,y)|^{(2+p)\beta}}{|x|^{\beta}} \cdot |x|^{\beta} |u| \big)^{\frac{1}{2+r}}\\
&\lesssim \big( \int_{\mathbb{R}^d \times \mathbb{T}} \frac{|u(t,x,y)|^{2+p}}{|x|}  \, dxdy \big)^{\frac{\beta}{2+r}}\cdot \big(\int (|x|^{\beta} |u|)^{\frac{1}{1-\beta}} \, dxdy \big)^{\frac{1-\beta}{2+r}}\\
&= \big( \int_{\mathbb{R}^d \times \mathbb{T}} \frac{|u(t,x,y)|^{2+p}}{|x|}  \, dxdy \big)^{\frac{\beta}{2+r}}\cdot \big(\|(|x|^{\beta} |u|)\|_{L_{x,y}^{\frac{1}{1-\beta}}} \big)^{\frac{1}{2+r}} \\
&\lesssim \big( \int_{\mathbb{R}^d \times \mathbb{T}} \frac{|u(t,x,y)|^{2+p}}{|x|}  \, dxdy \big)^{\frac{\beta}{2+r}}\cdot \big( \big\| \||\nabla|^s u\|_{L^2_x} \big\|_{L_{y}^{\frac{1}{1-\beta}}} \big)^{\frac{1}{2+r}} \\
&\lesssim \big( \int_{\mathbb{R}^d \times \mathbb{T}} \frac{|u(t,x,y)|^{2+p}}{|x|}  \, dxdy \big)^{\frac{\beta}{2+r}}\cdot \big(  \||\nabla_x|^s |\nabla_{y}|^{\tau}u\|_{L^2_xL_{y}^{2}}  \big)^{\frac{1}{2+r}} \\
&\lesssim \big( \int_{\mathbb{R}^d \times \mathbb{T}} \frac{|u(t,x,y)|^{2+p}}{|x|}  \, dxdy \big)^{\frac{\beta}{2+r}}.
\end{align}
We require the indices satisfy:
\begin{align}
&s+\tau \leq \sigma \quad \textmd{(regularity requirement from the energy conservation)}, \\
&(\beta-\frac{1}{2})+=\tau \quad \textmd{(Sobolev embedding in 1D)},\\
& \beta+s=\frac{d}{2}-d(1-\beta) \quad \textmd{(radial Sobolev embedding)},\\
& (2+p)\beta+1=2+r \quad \textmd{(the relation between $\beta$ and $r$)}.
\end{align}
We need to choose $\beta$ satisfies $\beta<\frac{1}{2}+\frac{\frac{1}{2}+\sigma}{d}$. Correspondingly, $r<\frac{2+4\sigma+dp+p+2p\sigma}{2d}$. That is the exponent requirement in the decay estimate \eqref{decay}.

We are now ready to prove \eqref{decay} by contradiction argument. If \eqref{decay} does not hold, using the estimate above, we deduce the existence of a sequence $\{t_n\}_{n} \rightarrow \infty$ and $\epsilon_0>0$ such that  
\begin{equation}
  \big( \int_{\mathbb{R}^d \times \mathbb{T}} \frac{|u(t_n,x,y)|^{2+p}}{|x|} \, dxdy \big)^{\frac{\beta}{2+r}}  \geq \|u(t_n,x,y)\|_{L^{2+r}_{x,y} (\R^d \times \T)} \geq \epsilon_0>0.
\end{equation}
Without loss of generality, we consider $t_n \rightarrow +\infty$. Similar as in \cite{TV2}, we get the existence of $T>0$ such that
\begin{equation}
  \inf_{n} \inf_{t\in (t_n,t_n+T)} \big( \int_{\mathbb{R}^d \times \mathbb{T}} \frac{|u(t,x,y)|^{2+p}}{|x|}  \, dxdy \big)^{\frac{\beta}{2+r}}  \geq \frac{\epsilon_0}{2}.    
\end{equation}
Notice that since $\{t_n\}_{n} \rightarrow +\infty$ then we can assume (modulo subsequence) that the
intervals $(t_n, t_n+T)$ are disjoint. In particular we have
\begin{align}
 \sum_n T(\frac{\epsilon_0}{2})^{\frac{2+r}{\beta}} &\leq \sum_n \int_{t_n}^{t_n+T} \big( \int_{\mathbb{R}^d \times \mathbb{T}} \frac{|u(t,x,y)|^{2+p}}{|x|}  \, dxdy \big)  dt\\
 &\leq  \int \big( \int_{\mathbb{R}^d \times \mathbb{T}} \frac{|u(t,x,y)|^{2+p}}{|x|}  \, dxdy \big) dt  
\end{align}
and hence we get a contradiction since the left hand side is divergent and the right
hand side is bounded by \eqref{Morawetz}.
\subsection{Step 3: Proof of the spacetime bound}
We aim to show,
\begin{equation}\label{est1}
  u \in L_t^{q_{\theta}}L_x^{r_{\theta}}H_{y}^{\frac{1}{2}+\delta}(\mathbb{R}_t\times \R^d \times \T)  
\end{equation}
and
\begin{equation}\label{est2}
  \lim_{t_1,t_2\rightarrow \infty} \parenthese{ \|u\|_{L_t^{l'}L_x^{m'}L^2_{y}([t_1,t_2] \times \R^d \times \T)}+\||\nabla_x|^{\sigma}(u)\|_{L_t^{l}L_x^{m}L^2_{y}([t_1,t_2] \times \R^d \times \T)}+\||\partial_{y}|^{\sigma}(u)\|_{L_t^{l}L_x^{m}L^2_{y}([t_1,t_2] \times \R^d \times \T)}}< \infty. 
\end{equation}

The above spacetime bounds are sufficient to show the scattering for \eqref{gNLS}. In this step, all spacetime norms are over $\mathbb{R}_t\times \R^d \times \T$ unless indicated otherwise. For example, we define
\begin{align*}
\|f(t)\|_{L^p_{t>t_0}} : = (\int_{t_0}^{\infty}|f(t)|^{p} \, dt)^{\frac{1}{p}}
\end{align*}
for any given time-dependent function $f(t)$, and similarly we can define $\|f(t)\|_{L^p_{t<t_0}}$. We note that we will apply an $H_{y}^{\frac{1}{2}+\delta}$ valued version of the critical analysis of \cite{cazenave2003semilinear}.

\begin{proof}
Using Strichartz estimates and the H\"older inequality,
\begin{equation}
\aligned
    \|u\|_{L_{t>t_0}^{q_{\theta}}L_x^{r_{\theta}}H_{y}^{\frac{1}{2}+\delta}} &\lesssim \|u_0\|_{H^{\sigma} (\R^d \times \T)}+\||u|^pu\|_{L_{t>t_0}^{\tilde{q_{\theta}}'}L_x^{\tilde{r_{\theta}}'}H_{y}^{\frac{1}{2}+\delta}} \\
    &\lesssim \|u_0\|_{H^{\sigma}(\R^d \times \T)}+\|u\|^{1+p}_{L_{t>t_0}^{(1+p)\tilde{q_{\theta}}'}L_x^{(1+p)\tilde{r_{\theta}}'}H_{y}^{\frac{1}{2}+\delta}} \\
    &\lesssim \|u_0\|_{H^{\sigma}(\R^d \times \T)}+\|u\|^{(1+p)\theta}_{L_{t>t_0}^{q_{\theta}}L_x^{r_{\theta}}H_{y}^{\frac{1}{2}+\delta}}\|u\|^{(1+p)(1-\theta)}_{L_{t>t_0}^{\infty}L_x^{\frac{pd}{2}}H_{y}^{\frac{1}{2}+\delta}}.
\endaligned    
\end{equation}
Similar to Lemma 2.5 in \cite{TV2} (this lemma is an analysis result which does not involve the nonlinear PDE structure so we can use it directly), based on the decay property \eqref{decay}, we can further obtain
\begin{equation}\label{decay2}
 \|u\|_{L_x^{\frac{pd}{2}}H_{y}^{\frac{1}{2}+\delta}}=o(1).
\end{equation}
Using the decay property \eqref{decay2}, we see for every $\epsilon>0$ there exists $t_0=t_0(\epsilon)>0$ such that
\begin{equation}
    \|u\|_{L_{t>t_0}^{q_{\theta}}L_x^{r_{\theta}}H_{y}^{\frac{1}{2}+\delta}} \leq C\|u_0\|_{H^2(\R^d \times \T)}+\epsilon \|u\|_{L_{t>t_0}^{q_{\theta}}L_x^{r_{\theta}}H_{y}^{\frac{1}{2}+\delta}}.
\end{equation}
We can now use the continuity argument to obtain
\begin{equation}
\|u\|_{L_{t>0}^{q_{\theta}}L_x^{r_{\theta}}H_{y}^{\frac{1}{2}+\delta}}<\infty.
\end{equation}
Similarly, we obtain $\|u\|_{L_{t<0}^{q_{\theta}}L_x^{r_{\theta}}H_{y}^{\frac{1}{2}+\delta}}<\infty $.

Now we consider the second estimate. We show $\||\partial_{y}|^{\sigma}(u)\|_{L_t^{l}L_x^{m}L^2_{y}}$, the other estimates are similar.  Using Strichartz estimate and the H\"older inequality,
\begin{equation}
\aligned
   \||\partial_{y}|^{\sigma}(u)\|_{L_{t>t_0}^{l}L_x^{m}L^2_{y}} &\lesssim \|u_0\|_{H^{\sigma}(\R^d \times \T)}+\||\partial_{y}|^{\sigma}(|u|^pu)\|_{L_{t>t_0}^{l'}L_x^{m'}L^2_{y}} \\
    &\lesssim \|u_0\|_{H^{\sigma}(\R^d \times \T)}+\||\partial_{y}|^{\sigma}(u)\|_{L_{t>t_0}^{l}L_x^{m}L^2_{y}}\|u\|^p_{L_{t>t_0}^{q_{\theta}}L^{r_{\theta}}_xL_{y}^{\infty}} \\
    &\lesssim \|u_0\|_{H^{\sigma}(\R^d \times \T)}+\||\partial_{y}|^{\sigma}(u)\|_{L_{t>t_0}^{l}L_x^{m}L^2_{y}}\|u\|^p_{L_{t>t_0}^{q_{\theta}}L^{r_{\theta}}_xH_{y}^{\frac{1}{2}+\delta}}.
\endaligned    
\end{equation}
We conclude by choosing $t_0$ large enough and by recalling \eqref{est1}.

For FNLS, due to the Strichartz estimates and Sobolev embedding, we choose the indices satisfying, 
\begin{equation}
   s+\frac{1}{2}+\delta \leq \sigma, \quad \textmd{(regularity requirement from the energy conservation)} 
\end{equation}
\begin{equation}
    \frac{2\sigma}{q_{\theta}}+\frac{d}{r_{\theta}}=\frac{d}{2}-s,\quad \frac{2\sigma}{q_{\theta}}+\frac{d}{\tilde{r}_{\theta}}+\frac{2\sigma}{\tilde{q}_{\theta}}+\frac{d}{r_{\theta}}=d, \quad \textmd{(Strichartz exponent relations)}
\end{equation}
\begin{equation}
    \frac{1}{(p+1)\tilde{q_{\theta}}'}=\frac{\theta}{q_{\theta}}, \quad \frac{1}{(p+1)\tilde{r_{\theta}}'}=\frac{\theta}{r_{\theta}}+\frac{2(1-\theta)}{pd},\quad \textmd{(the H\"older inequality, or say, interpolation)}
\end{equation}
and
\begin{equation}
    \frac{2\sigma}{l}+\frac{d}{m}=\frac{d}{2}, \quad \frac{1}{m'}=\frac{1}{m}+\frac{p}{r_{\theta}},\quad \frac{1}{l'}=\frac{1}{l}+\frac{p}{q_{\theta}} \quad \textmd{(Strichartz exponent relations and the H\"older inequality)}.
\end{equation}
\end{proof}

\subsection{Step 4: Proof of the scattering asymptotics}
In fact by using the integral equation, it is sufficient to prove that
\begin{equation}
 \lim_{t_1,t_2\rightarrow \infty}\norm{\int_{t_1}^{t_2}e^{-is((-\Delta_{x})^{\sigma}-\partial_{y}^{2\sigma})}(|u|^pu) \, ds  }_{H^{\sigma}_{x,y}  (\R^d \times \mathbb{T})}=0   .
\end{equation}

Moreover, using Strichartz estimates, we only need to show,
\begin{align}
  \lim_{t_1,t_2\rightarrow \infty} \left( \||u|^pu\|_{L_t^{l'}L_x^{m'}L^2_{y}([t_1,t_2]\times \mathbb{R}^d \times \mathbb{T})}+\||\nabla_x|^{\sigma}(|u|^pu)\|_{L_t^{l'}L_x^{m'}L^2_{y}([t_1,t_2]\times \mathbb{R}^d \times \mathbb{T})} \right.\\
\left.  +\||\partial_{y}|^{\sigma}(|u|^pu)\|_{L_t^{l'}L_x^{m'}L^2_{y}([t_1,t_2]\times \mathbb{R}^d \times \mathbb{T})} \right)=0. 
\end{align}
Noticing the two established estimates, the above limit follows in a straightforward way. Thus we proved scattering in the energy space.
\section{Further remarks}\label{sec Future}
In this section, we make a few more remarks on this research line, i.e. \emph{long time dynamics for dispersive equations on waveguide manifolds}. As mentioned in the introduction, this area has been developed a lot in recent decades. The authors are interested in this research line for several years. Though many theories/tools/results have been established, there are still many interesting open questions left. We list some interesting related problems in this line for interested readers.

1. \emph{The critical regime.} The cases we are considering in this paper are of a `double subcritical' nature \eqref{gNLS}. In fact, it is also quite interesting to consider the scattering theory for the critical regime. For example, 
\begin{equation}
     (i\partial_t+(-\Delta_{x})^{\sigma} + (-\partial_{y}^{2})^{\sigma})u= \mu |u|^{\frac{4\sigma}{d}}u,\quad u(0,x,y)= u_0(x,y)\in H^{\sigma}(\mathbb{R}^d\times \mathbb{T}),
\end{equation}
and
\begin{equation}
   (i\partial_t+(-\Delta_{x})^{\sigma}+ (-\partial_{y}^{2})^{\sigma})u= \mu |u|^{\frac{4\sigma}{d+1-2\sigma}}u,\quad u(0,x,y)= u_0(x,y) \in H^{\sigma}(\mathbb{R}^d\times \mathbb{T}).
\end{equation} 
The first one is of mass-critical nature and the second one is of energy-critical nature. New techniques are needed including function spaces, profile decomposition, profile approximations and even resonant systems. See \cite{CGZ,HP,Z1,Z2} for the NLS case. 

2. \emph{Improvements/generalizations for Theorem \ref{mainthm}.} One may also try to remove the radial assumption in Theorem \ref{mainthm} or consider \eqref{gNLS} on more general waveguide manifolds $\mathbb{R}^n \times  \mathbb{T}^m$. However, when $m \geq 2$, one can at most show the global well-posedness since the scattering is not expected to hold. Another point is to consider the cases $d=1,2$. Some more techniques are required, such as the Morawetz-type estimate in 1D and 2D (see Section \ref{sec Morawetz}).  

One may consider other problems for \eqref{gNLS} such as growth of Sobolev norms (weak turbulence) or low regularity type results (see \cite{ZhaoZheng} and the references therein).

3. \emph{Scattering for focusing NLS/4NLS/gNLS on waveguide manifolds.} The results discussed in this paper concern mainly the defocusing case. In general, large data scattering for the focusing NLS (or other dispersive equations) on waveguides is comparably less understood than the defocusing case. Threshold assumptions are necessary and new ingredients are needed to handle this type of problems. See \cite{YYZ} for a recent global well-posedness result, see \cite{dodson2019global,DHR,KM06,killip2010focusing} for the Euclidean results and see \cite{cheng2022global,luo2022large} for some very recent scattering result.

4. \emph{Critical NLS on higher dimensional waveguide manifolds.} For critical NLS (or other dispersive models) on waveguide manifolds, most of the models are lower dimensional (with no higher than four whole dimensions), which leads to quintic or cubic nonlinearity. This gives one advantage to applying function spaces to deal with nonlinearity. In general, the difficulty of the critical NLS problem on $\mathbb{R}^m \times \mathbb{T}^n$ increases if the dimension $m+n$ is increased or if the number $m$ of copies of $\mathbb{R}$ is decreased (which is concluded in \cite{IPRT3}). There are no large data global results for critical NLS on waveguide manifolds with at least $5$ whole dimensions, to the best knowledge of the authors. Moreover, the Hartree analogues are also less understood. 

5. \emph{NLS on other product spaces.} Instead of waveguide manifolds, one may consider dispersive equations on other types of product spaces, for example, $\mathbb{R}^d\times \mathbb{S}^n$ where $\mathbb{S}^n$ are n-dimensional spheres ($\mathbb{S}^n$ can be replaced by other manifolds). See \cite{pausader2014global} for a global well-posedness result of NLS on pure spheres. In this regime, NLS may be a good model to start with. One can also replace $\mathbb{S}^n$ by other manifolds.

\bibliographystyle{plain}
\bibliography{BG4NLS}

\begin{thebibliography}{10}

\bibitem{ACS}
P.~Antonelli, R.~Carles, and J.~Drumond~Silva.
\newblock Scattering for nonlinear {S}chr\"{o}dinger equation under partial
  harmonic confinement.
\newblock {\em Comm. Math. Phys.}, 334(1):367--396, 2015.

\bibitem{BGX}
H.~Bahouri, P.~G\'{e}rard, and C.-J. Xu.
\newblock Espaces de {B}esov et estimations de {S}trichartz
  g\'{e}n\'{e}ralis\'{e}es sur le groupe de {H}eisenberg.
\newblock {\em J. Anal. Math.}, 82:93--118, 2000.

\bibitem{Balakrishnan}
A.~V. Balakrishnan.
\newblock Fractional powers of closed operators and the semigroups generated by
  them.
\newblock {\em Pacific J. Math.}, 10:419--437, 1960.

\bibitem{Barron}
A.~Barron.
\newblock On global-in-time {S}trichartz estimates for the semiperiodic
  {S}chr\"{o}dinger equation.
\newblock {\em Anal. PDE}, 14(4):1125--1152, 2021.

\bibitem{bertoin}
J.~Bertoin.
\newblock {\em L\'{e}vy processes}, volume 121 of {\em Cambridge Tracts in
  Mathematics}.
\newblock Cambridge University Press, Cambridge, 1996.

\bibitem{BHL}
T.~Boulenger, D.~Himmelsbach, and E.~Lenzmann.
\newblock Blowup for fractional {NLS}.
\newblock {\em J. Funct. Anal.}, 271(9):2569--2603, 2016.

\bibitem{BD}
J.~Bourgain and C.~Demeter.
\newblock The proof of the {$l^2$} decoupling conjecture.
\newblock {\em Ann. of Math. (2)}, 182(1):351--389, 2015.

\bibitem{BGT}
N.~Burq, P.~G\'{e}rard, and N.~Tzvetkov.
\newblock Strichartz inequalities and the nonlinear {S}chr\"{o}dinger equation
  on compact manifolds.
\newblock {\em Amer. J. Math.}, 126(3):569--605, 2004.

\bibitem{carmona}
R.~Carmona, W.~C. Masters, and B.~Simon.
\newblock Relativistic {S}chr\"{o}dinger operators: asymptotic behavior of the
  eigenfunctions.
\newblock {\em J. Funct. Anal.}, 91(1):117--142, 1990.

\bibitem{cazenave2003semilinear}
T.~Cazenave.
\newblock {\em Semilinear {S}chr\"{o}dinger equations}, volume~10 of {\em
  Courant Lecture Notes in Mathematics}.
\newblock New York University, Courant Institute of Mathematical Sciences, New
  York; American Mathematical Society, Providence, RI, 2003.

\bibitem{Guo5}
X.~Cheng, C.~Guo, Z.~Guo, X.~Liao, and J.~Shen.
\newblock Scattering of the three-dimensional cubic nonlinear {S}chr\"{o}dinger
  equation with partial harmonic potentials.
\newblock {\em arXiv preprint arXiv:2105.02515}, 2021.

\bibitem{cheng2022global}
X.~Cheng, Z.~Guo, G.~Hwang, and H.~Yoon.
\newblock Global well-posedness and scattering of the two dimensional cubic
  focusing nonlinear {S}chr\"{o}dinger system.
\newblock {\em arXiv preprint arXiv:2202.10757}, 2022.

\bibitem{CGZ}
X.~Cheng, Z.~Guo, and Z.~Zhao.
\newblock On scattering for the defocusing quintic nonlinear {S}chr\"{o}dinger
  equation on the two-dimensional cylinder.
\newblock {\em SIAM J. Math. Anal.}, 52(5):4185--4237, 2020.

\bibitem{CZZ}
X.~Cheng, Z.~Zhao, and J.~Zheng.
\newblock Well-posedness for energy-critical nonlinear {S}chr\"{o}dinger
  equation on waveguide manifold.
\newblock {\em J. Math. Anal. Appl.}, 494(2):Paper No. 124654, 14, 2021.

\bibitem{CGKL13}
Y.~Cho, G.~Hwang, S.~Kwon, and S.~Lee.
\newblock Profile decompositions and blowup phenomena of mass critical
  fractional {S}chr\"{o}dinger equations.
\newblock {\em Nonlinear Anal.}, 86:12--29, 2013.

\bibitem{CHKL}
Y.~Cho, G.~Hwang, S.~Kwon, and S.~Lee.
\newblock Well-posedness and ill-posedness for the cubic fractional
  {S}chr\"{o}dinger equations.
\newblock {\em Discrete Contin. Dyn. Syst.}, 35(7):2863--2880, 2015.

\bibitem{CKSTT}
J.~Colliander, M.~Keel, G.~Staffilani, H.~Takaoka, and T.~Tao.
\newblock Global well-posedness and scattering for the energy-critical
  nonlinear {S}chr\"{o}dinger equation in {$\Bbb R^3$}.
\newblock {\em Ann. of Math. (2)}, 167(3):767--865, 2008.

\bibitem{daubechies}
I.~Daubechies and E.~H. Lieb.
\newblock One-electron relativistic molecules with {C}oulomb interaction.
\newblock {\em Comm. Math. Phys.}, 90(4):497--510, 1983.

\bibitem{DET}
S.~Demirbas, M.~B. Erdo\u{g}an, and N.~Tzirakis.
\newblock Existence and uniqueness theory for the fractional {S}chr\"{o}dinger
  equation on the torus.
\newblock In {\em Some topics in harmonic analysis and applications}, volume~34
  of {\em Adv. Lect. Math. (ALM)}, pages 145--162. Int. Press, Somerville, MA,
  2016.

\bibitem{Dinh}
V.~D. Dinh.
\newblock Strichartz estimates for the fractional {S}chr\"{o}dinger and wave
  equations on compact manifolds without boundary.
\newblock {\em J. Differential Equations}, 263(12):8804--8837, 2017.

\bibitem{dodson2019global}
B.~Dodson.
\newblock Global well-posedness and scattering for the focusing, cubic
  {S}chr\"{o}dinger equation in dimension {$d=4$}.
\newblock {\em Ann. Sci. \'{E}c. Norm. Sup\'{e}r. (4)}, 52(1):139--180, 2019.

\bibitem{DHR}
T.~Duyckaerts, J.~Holmer, and S.~Roudenko.
\newblock Scattering for the non-radial 3{D} cubic nonlinear {S}chr\"{o}dinger
  equation.
\newblock {\em Math. Res. Lett.}, 15(6):1233--1250, 2008.

\bibitem{fan2021decay}
C.~Fan and Z.~Zhao.
\newblock Decay estimates for nonlinear {S}chr\"{o}dinger equations.
\newblock {\em Discrete Contin. Dyn. Syst.}, 41(8):3973--3984, 2021.

\bibitem{FLS1}
R.~L. Frank, E.~H. Lieb, and R.~Seiringer.
\newblock Stability of relativistic matter with magnetic fields for nuclear
  charges up to the critical value.
\newblock {\em Comm. Math. Phys.}, 275(2):479--489, 2007.

\bibitem{FLS2}
R.~L. Frank, E.~H. Lieb, and R.~Seiringer.
\newblock Hardy-{L}ieb-{T}hirring inequalities for fractional {S}chr\"{o}dinger
  operators.
\newblock {\em J. Amer. Math. Soc.}, 21(4):925--950, 2008.

\bibitem{GG}
Patrick G\'{e}rard and Sandrine Grellier.
\newblock L'\'{e}quation de {S}zego cubique.
\newblock In {\em S\'{e}minaire: \'{E}quations aux {D}\'{e}riv\'{e}es
  {P}artielles. 2008--2009}, S\'{e}min. \'{E}qu. D\'{e}riv. Partielles, pages
  Exp. No. II, 19. \'{E}cole Polytech., Palaiseau, 2010.

\bibitem{GH11}
B.~Guo and Z.~Huo.
\newblock Global well-posedness for the fractional nonlinear {S}chr\"{o}dinger
  equation.
\newblock {\em Comm. Partial Differential Equations}, 36(2):247--255, 2011.

\bibitem{GSWZ}
Z.~Guo, Y.~Sire, Y.~Wang, and L.~Zhao.
\newblock On the energy-critical fractional {S}chr\"{o}dinger equation in the
  radial case.
\newblock {\em Dyn. Partial Differ. Equ.}, 15(4):265--282, 2018.

\bibitem{GW}
Z.~Guo and Y.~Wang.
\newblock Improved {S}trichartz estimates for a class of dispersive equations
  in the radial case and their applications to nonlinear {S}chr\"{o}dinger and
  wave equations.
\newblock {\em J. Anal. Math.}, 124:1--38, 2014.

\bibitem{HP}
Z.~Hani and B.~Pausader.
\newblock On scattering for the quintic defocusing nonlinear {S}chr\"{o}dinger
  equation on {$\Bbb R\times\Bbb T^2$}.
\newblock {\em Comm. Pure Appl. Math.}, 67(9):1466--1542, 2014.

\bibitem{HT}
Z.~Hani and L.~Thomann.
\newblock Asymptotic behavior of the nonlinear {S}chr\"{o}dinger equation with
  harmonic trapping.
\newblock {\em Comm. Pure Appl. Math.}, 69(9):1727--1776, 2016.

\bibitem{HTT1}
S.~Herr, D.~Tataru, and N.~Tzvetkov.
\newblock Global well-posedness of the energy-critical nonlinear
  {S}chr\"{o}dinger equation with small initial data in {$H^1(\Bbb T^3)$}.
\newblock {\em Duke Math. J.}, 159(2):329--349, 2011.

\bibitem{HTT2}
S.~Herr, D.~Tataru, and N.~Tzvetkov.
\newblock Strichartz estimates for partially periodic solutions to
  {S}chr\"{o}dinger equations in {$4d$} and applications.
\newblock {\em J. Reine Angew. Math.}, 690:65--78, 2014.

\bibitem{HS}
Y.~Hong and Y.~Sire.
\newblock On fractional {S}chr\"{o}dinger equations in {S}obolev spaces.
\newblock {\em Commun. Pure Appl. Anal.}, 14(6):2265--2282, 2015.

\bibitem{IPT3}
A.~D. Ionescu and B.~Pausader.
\newblock The energy-critical defocusing {NLS} on {$\Bbb T^3$}.
\newblock {\em Duke Math. J.}, 161(8):1581--1612, 2012.

\bibitem{IPRT3}
A.~D. Ionescu and B.~Pausader.
\newblock Global well-posedness of the energy-critical defocusing {NLS} on
  {$\Bbb R\times \Bbb T^3$}.
\newblock {\em Comm. Math. Phys.}, 312(3):781--831, 2012.

\bibitem{IP14}
A.~D. Ionescu and F.~Pusateri.
\newblock Nonlinear fractional {S}chr\"{o}dinger equations in one dimension.
\newblock {\em J. Funct. Anal.}, 266(1):139--176, 2014.

\bibitem{Kato}
T.~Kato.
\newblock On nonlinear {S}chr\"{o}dinger equations. {II}. {$H^s$}-solutions and
  unconditional well-posedness.
\newblock {\em J. Anal. Math.}, 67:281--306, 1995.

\bibitem{KM06}
C.~E. Kenig and F.~Merle.
\newblock Global well-posedness, scattering and blow-up for the
  energy-critical, focusing, non-linear {S}chr\"{o}dinger equation in the
  radial case.
\newblock {\em Invent. Math.}, 166(3):645--675, 2006.

\bibitem{KV1}
R.~Killip and M.~Vi\c{s}an.
\newblock Scale invariant {S}trichartz estimates on tori and applications.
\newblock {\em Math. Res. Lett.}, 23(2):445--472, 2016.

\bibitem{killip2010focusing}
R.~Killip and M.~Visan.
\newblock The focusing energy-critical nonlinear {S}chr\"{o}dinger equation in
  dimensions five and higher.
\newblock {\em Amer. J. Math.}, 132(2):361--424, 2010.

\bibitem{KLR13}
J.~Krieger, E.~Lenzmann, and P.~Rapha\"{e}l.
\newblock Nondispersive solutions to the {$L^2$}-critical half-wave equation.
\newblock {\em Arch. Ration. Mech. Anal.}, 209(1):61--129, 2013.

\bibitem{Laskin2}
N.~Laskin.
\newblock Fractals and quantum mechanics.
\newblock {\em Chaos}, 10(4):780--790, 2000.

\bibitem{Laskin}
N.~Laskin.
\newblock Fractional quantum mechanics and {L}\'{e}vy path integrals.
\newblock {\em Phys. Lett. A}, 268(4-6):298--305, 2000.

\bibitem{Laskin3}
N.~Laskin.
\newblock Fractional {S}chr\"{o}dinger equation.
\newblock {\em Phys. Rev. E (3)}, 66(5):056108, 7, 2002.

\bibitem{LiebYau1}
E.~H. Lieb and H.-T. Yau.
\newblock The {C}handrasekhar theory of stellar collapse as the limit of
  quantum mechanics.
\newblock {\em Comm. Math. Phys.}, 112(1):147--174, 1987.

\bibitem{LiebYau2}
E.~H. Lieb and H.-T. Yau.
\newblock The stability and instability of relativistic matter.
\newblock {\em Comm. Math. Phys.}, 118(2):177--213, 1988.

\bibitem{Lon15}
S.~Longhi.
\newblock Fractional schr{\"o}dinger equation in optics.
\newblock {\em Optics letters}, 40(6):1117--1120, 2015.

\bibitem{luo2022large}
Y.~Luo.
\newblock Large data global well-posedness and scattering for the focusing
  cubic nonlinear {S}chr\"{o}dinger equation on $\mathbb{R}^2\times
  \mathbb{T}$.
\newblock {\em arXiv preprint arXiv:2202.10219}, 2022.

\bibitem{pausader2014global}
B.~Pausader, N.~Tzvetkov, and X.~Wang.
\newblock Global regularity for the energy-critical {NLS} on {$\Bbb{S}^3$}.
\newblock {\em Ann. Inst. H. Poincar\'{e} Anal. Non Lin\'{e}aire},
  31(2):315--338, 2014.

\bibitem{PV}
F.~Planchon and L.~Vega.
\newblock Bilinear virial identities and applications.
\newblock {\em Ann. Sci. \'{E}c. Norm. Sup\'{e}r. (4)}, 42(2):261--290, 2009.

\bibitem{SW21}
J.-C. Saut and Y.~Wang.
\newblock Global dynamics of small solutions to the modified fractional
  {K}orteweg--de {V}ries and fractional cubic nonlinear {S}chr\"{o}dinger
  equations.
\newblock {\em Comm. Partial Differential Equations}, 46(10):1851--1891, 2021.

\bibitem{Schippa}
R.~Schippa.
\newblock On {S}trichartz estimates from $l^2$-decoupling and applications.
\newblock pages 279--289, 2020.

\bibitem{St13}
B.A. Stickler.
\newblock Potential condensed-matter realization of space-fractional quantum
  mechanics: The one-dimensional l{\'e}vy crystal.
\newblock {\em Physical Review E}, 88(1):012120, 2013.

\bibitem{ST20}
C.~Sun and N.~Tzvetkov.
\newblock Gibbs measure dynamics for the fractional {NLS}.
\newblock {\em SIAM Journal on Mathematical Analysis}, 52(5):4638--4704, 2020.

\bibitem{ST21}
C.~Sun and N.~Tzvetkov.
\newblock Refined probabilistic global well-posedness for the weakly dispersive
  {NLS}.
\newblock {\em Nonlinear Analysis. Theory, Methods \& Applications. An
  International Multidisciplinary Journal}, 213:Paper No. 112530, 91, 2021.

\bibitem{SWTZ18}
C.~Sun, H.~Wang, X.~Yao, and J.~Zheng.
\newblock Scattering below ground state of focusing fractional nonlinear
  {S}chr\"{o}dinger equation with radial data.
\newblock {\em Discrete Contin. Dyn. Syst.}, 38(4):2207--2228, 2018.

\bibitem{SY2}
M.~Sy and X.~Yu.
\newblock Global well-posedness and long-time behavior of the fractional {NLS}.
\newblock {\em Stochastics and Partial Differential Equations: Analysis and
  Computations}, 2021.

\bibitem{SY1}
M.~Sy and X.~Yu.
\newblock Global well-posedness for the cubic fractional {NLS} on the unit
  disk.
\newblock {\em Nonlinearity}, 35(4):2020--2072, 2022.

\bibitem{Taobook}
T.~Tao.
\newblock {\em Nonlinear dispersive equations}, volume 106 of {\em CBMS
  Regional Conference Series in Mathematics}.
\newblock Published for the Conference Board of the Mathematical Sciences,
  Washington, DC; by the American Mathematical Society, Providence, RI, 2006.
\newblock Local and global analysis.

\bibitem{TVZ}
T.~Tao, M.~Visan, and X.~Zhang.
\newblock Global well-posedness and scattering for the defocusing mass-critical
  nonlinear {S}chr\"{o}dinger equation for radial data in high dimensions.
\newblock {\em Duke Math. J.}, 140(1):165--202, 2007.

\bibitem{TV1}
N.~Tzvetkov and N.~Visciglia.
\newblock Small data scattering for the nonlinear {S}chr\"{o}dinger equation on
  product spaces.
\newblock {\em Comm. Partial Differential Equations}, 37(1):125--135, 2012.

\bibitem{TV2}
N.~Tzvetkov and N.~Visciglia.
\newblock Well-posedness and scattering for nonlinear {S}chr\"{o}dinger
  equations on {$\Bbb{R}^d\times\Bbb{T}$} in the energy space.
\newblock {\em Rev. Mat. Iberoam.}, 32(4):1163--1188, 2016.

\bibitem{YYZ2}
X.~Yu, H.~Yue, and Z.~Zhao.
\newblock Global well-posedness and scattering for fourth-order
  {S}chr\"{o}dinger equations on waveguide manifolds.
\newblock {\em arXiv preprint arXiv:2111.09651}, 2021.

\bibitem{YYZ}
X.~Yu, H.~Yue, and Z.~Zhao.
\newblock Global {W}ell-posedness for the focusing cubic {NLS} on the product
  space {$\Bbb{R} \times \Bbb{T}^3$}.
\newblock {\em SIAM J. Math. Anal.}, 53(2):2243--2274, 2021.

\bibitem{yu2022decay}
X.~Yu, H.~Yue, and Z.~Zhao.
\newblock On the decay property of the cubic fourth-order {S}chr\"{o}dinger
  equation.
\newblock {\em arXiv preprint arXiv:2201.00515}, 2022.

\bibitem{Z1}
Z.~Zhao.
\newblock Global well-posedness and scattering for the defocusing cubic
  {S}chr\"{o}dinger equation on waveguide {$\Bbb R^2\times\Bbb T^2$}.
\newblock {\em J. Hyperbolic Differ. Equ.}, 16(1):73--129, 2019.

\bibitem{Z2}
Z.~Zhao.
\newblock On scattering for the defocusing nonlinear {S}chr\"{o}dinger equation
  on waveguide {$\Bbb R^m\times \Bbb T$} (when {$m = 2,3$}).
\newblock {\em J. Differential Equations}, 275:598--637, 2021.

\bibitem{ZhaoZheng}
Z.~Zhao and J.~Zheng.
\newblock Long time dynamics for defocusing cubic nonlinear {S}chr\"{o}dinger
  equations on three dimensional product space.
\newblock {\em SIAM J. Math. Anal.}, 53(3):3644--3660, 2021.

\end{thebibliography}

\end{document}